\documentclass{article}
\usepackage{geometry}
\usepackage{authblk}
\usepackage[utf8]{inputenc}
\usepackage[T1]{fontenc}
\usepackage{amsmath}
\usepackage{todonotes}
\usepackage{mathtools}
\usepackage{amssymb}
\usepackage{amsthm}
\usepackage[hyphens]{url}
\usepackage{hyperref}
\usepackage[hyphenbreaks]{breakurl}
\usepackage{cleveref}
\usepackage{tabularx}
\usepackage{mathdots}
\usepackage{mathrsfs}
\usepackage{color}
\usepackage{xcolor}
\usepackage{comment}
\usepackage{tikz}
\usetikzlibrary{arrows.meta}
\usepackage{pgfplots} 
\usepackage{tcolorbox}
\usepackage{pdfpages}
\usepackage[width=14cm,format=hang]{caption}
\usepackage{subcaption}
\usepackage{enumitem}

\newcommand*\samethanks[1][\value{footnote}]{\footnotemark[#1]}
\newcommand{\mres}{\mathbin{\vrule height 1.6ex depth 0pt width 0.13ex\vrule height 0.13ex depth 0pt width 1.3ex}}
\newcommand{\Ha}{\mathcal{H}^1}
\newcommand{\Le}{\mathcal{L}}

\newcommand{\R}{\mathbb{R}}
\newcommand{\ws}{\stackrel{*}{\rightharpoonup}}
\newcommand{\F}{\mathcal{F}}
\newcommand{\G}{\mathcal{G}}
\renewcommand{\d}{\mathrm{d}}
\renewcommand{\div}{\mathrm{div}}

\newtheorem{corr}{Corollary}[subsection]
\newtheorem{prop}[corr]{Proposition}
\newtheorem{lem}[corr]{Lemma}
\newtheorem{rem}[corr]{Remark}
\newtheorem{examp}[corr]{Example}
\newtheorem{defin}[corr]{Definition}
\newtheorem{assump}[corr]{Assumption}
\newtheorem{conv}[corr]{Convention}
\newtheorem{thm}[corr]{Theorem}

\crefname{examp}{example}{examples}
\crefname{prop}{proposition}{propositions}
\crefname{lem}{lemma}{lemmas}
\crefname{defin}{definition}{definitions}
\crefname{assump}{assumption}{assumptions}
\crefname{figure}{figure}{figures}
\crefname{thm}{theorem}{theorems}

\newcommand{\notinclude}[1]{}

\newcolumntype{x}[1]{!{\centering\arraybackslash\vrule width #1}}
\makeatletter
\def\hlinewd#1{%
\noalign{\ifnum0=`}\fi\hrule \@height #1 \futurelet
\reserved@a\@xhline}
\makeatother

\usepackage[backend=biber,
maxnames=5,
style=alphabetic,
]{biblatex}
\nocite{*}

\title{Duality in branched transport and urban planning}
\author{Julius Lohmann\thanks{Institute for Numerical and Applied Mathematics, University of Münster, Einsteinstraße 62, 48149 Münster, Germany}\hspace{0.35em}\thanks{juliuslohmann@uni-muenster.de}\qquad Bernhard Schmitzer\thanks{Institute for Computer Science, University of Göttingen, Goldschmidtstraße 7, 37077 Göttingen, Germany. schmitzer@cs.uni-goettingen.de}\qquad Benedikt Wirth\samethanks[1]\hspace{0.35em}\thanks{benedikt.wirth@uni-muenster.de}}

\begin{document}
	\maketitle
	\begin{abstract}
		In recent work \cite[Theorem 1.3.4]{LSW} we have shown the equivalence of the widely used nonconvex (generalized) branched transport problem with a shape optimization problem of a street or railroad network, known as (generalized) urban planning problem.
		The argument was solely based on an explicit construction and characterization of competitors.
		In the current article we instead analyse the dual perspective associated with both problems.
		In more detail, the shape optimization problem involves the Wasserstein distance between two measures with respect to a metric depending on the street network.
		%In the latter shape optimization problem one seeks an optimal street network, which is given by a certain $1$-dimensional set equipped with a friction coefficient.
		%The generalized urban metric describes the cost to travel on and outside the network.
		We show a Kantorovich\textendash Rubinstein formula for Wasserstein distances on such street networks under mild assumptions. Further, we provide a Beckmann formulation for such Wasserstein distances under assumptions which generalize our previous result in \cite{LSW}. As an application we then give an alternative, duality-based proof of the equivalence of both problems under a growth condition on the transportation cost,
		which reveals that urban planning and branched transport can both be viewed as two bilinearly coupled convex optimization problems.

	\end{abstract}
	
	\tableofcontents
	
	\section{Introduction}
	Classical optimal transport takes two probability measures $\mu_+,\mu_-\in\mathcal P(\mathcal C)$ on some compact domain $\mathcal{C}\subset\R^n$
	and computes the optimal assignment $\pi$ between them,
	where $\pi$ is a probability measure on $\mathcal{C}\times\mathcal{C}$ with the interpretation that $\pi$ at $(x,y)$ represents the mass transported from $x$ to $y$.
	Optimality is here assessed with respect to the cost $\int_{\mathcal{C}\times\mathcal{C}}c(x,y)\,\d\pi(x,y)$ for some given transportation cost $c$.
	The special case $c(x,y)=|x-y|$ yields the so-called Wasserstein transport, inducing the Wasserstein-1 distance $W_1(\mu_+,\mu_-)$ between both measures.
	%\todo[inline,color=green]{Always ``Wasserstein-1 distance''.}
	An important reformulation of the Wasserstein transport is the so-called Beckmann formulation, obtained by dualizing the problem twice as we quickly recapitulate (rather formally) below:
	Introducing Lagrange multipliers $\phi,\psi$ for the constraint of $\pi$ being an assignment between $\mu_+$ and $\mu_-$, one obtains
	\begin{align*}
	W_1(\mu_+,\mu_-)
	&=\inf_{\pi\in\mathcal P(\mathcal C)}\left\{\int_{\mathcal{C}\times\mathcal{C}}|x-y|\,\d\pi(x,y)\,\middle|\,\pi(\cdot,\mathcal C)=\mu_+,\pi(\mathcal C,\cdot)=\mu_-\right\}\\
	&=\inf_{\pi\in\mathcal P(\mathcal C)}\sup_{\phi,\psi\in C(\mathcal C)}\int_{\mathcal{C}\times\mathcal{C}}|x-y|\,\d\pi(x,y)
	+\int_{\mathcal C}\phi(x)\,\d(\mu_+(x)-\pi(x,\mathcal C))
	+\int_{\mathcal C}\psi(y)\,\d(\mu_-(y)-\pi(\mathcal C,y)).\\
	\intertext{By classical Lagrange or Fenchel\textendash Rockafellar duality one can swap $\inf$ and $\sup$ and then minimize for $\pi$ to arrive at}
	%&=\sup_{\phi,\psi\in C(\mathcal C)}\inf_{\pi\in\mathcal P(\mathcal C)}\int_{\mathcal{C}\times\mathcal{C}}|x-y|\,\d\pi(x,y)
	%+\int_{\mathcal C}\phi(x)\,\d(\mu_+(x)-\pi(x,\mathcal C))
	%+\int_{\mathcal C}\psi(y)\,\d(\mu_-(y)-\pi(\mathcal C,y))\\
	&=\sup\left\{\int_{\mathcal C}\phi\,\d\mu_++\int_{\mathcal C}\psi\,\d\mu_-\,\middle|\,\phi,\psi\in C(\mathcal C),\,\phi(x)+\psi(y)\leq|x-y|\,\forall x,y\in\mathcal C\right\}.\\
	\intertext{Now it is not difficult to show that in the optimum $\psi$ has to equal $-\phi$ so that the above turns into the so-called Kantorovich\textendash Rubinstein formula}
	&=\sup\left\{\int_{\mathcal C}\phi\,\d(\mu_+-\mu_-)\,\middle|\,\phi\in C^{0,1}(\mathcal C)\text{ with }|\nabla\phi|\leq1\text{ a.e. in }\mathcal{C}\right\},\\
	\intertext{where we exploited that the constraint turns into $1$-Lipschitz continuity of $\phi$, which can equivalently be expressed by constraining $\nabla\phi$.
	Introducing a vector-valued Radon measure $\mathcal F\in\mathcal M^n(\mathcal C)$ on $\mathcal C$ as Lagrange multiplier for the new constraint we get}
	&=\sup_{\phi\in C(\mathcal C)}\inf_{\mathcal F\in\mathcal M^n(\mathcal C)}\int_{\mathcal C}\phi\,\d(\mu_+-\mu_-)+\int_{\mathcal C}\nabla\phi\cdot\d\mathcal F+|\mathcal F|(\mathcal C),\\
	\intertext{where $|\mathcal F|(\mathcal C)$ denotes the total variation of $\mathcal F$.
	An integration by parts now moves the derivative from $\phi$ onto $\mathcal F$.
	Furthermore, again by Lagrange or Fenchel\textendash Rockafellar duality we can swap $\inf$ and $\sup$ and then minimize explicitly over $\phi$, which finally yields the Beckmann formulation of the Wasserstein-1 distance,}
	%&=\inf_{\mathcal F\in\mathcal M(\mathcal C)^n}\sup_{\phi\in C(\mathcal C)}\int_{\mathcal C}\phi\,\d(\mu_+-\mu_-)+\int_{\partial\mathcal C}\phi n\cdot\d\mathcal F-\int_{\mathcal C}\phi\,\d\div(\mathcal F)+|\mathcal F|(\mathcal C)\\
	&=\inf\left\{|\mathcal F|(\mathcal C)\,\middle|\,\div (\mathcal F)=\mu_+-\mu_-\right\}.
	\end{align*}
	\begin{comment}
	If the cost $c(x,y)=|x-y|$ is replaced by the distance between $x$ and $y$ induced by a smooth Riemannian metric, the same simple argument applies.
	However, in some applications metrics $c(x,y)$ arise for which the above procedure becomes substantially more technical.
	These metrics for instance stem from a given street network or more generally a countably $1$-rectifiable set $S\subset\mathcal C$ on which the transport is cheaper than on its complement.
	In such situations neither the strong duality exploited at the end nor the characterization of $1$-Lipschitz functions via their gradients are straightforward.
	\end{comment}
	In this article we aim to derive the Kantorovich\textendash Rubinstein and the Beckmann formula for the setting where $c(x,y)$ is replaced by a (pseudo-)metric which stems from a street network or more generally a countably $1$-rectifiable set $S\subset\mathcal C$ on which the transport is cheaper than on its complement.
	These formulations are needed to understand the relation between the so-called \emph{urban planning} and \emph{branched transport problem}.

	The \emph{urban planning problem} was proposed by Brancolini and Buttazzo \cite{BB} as well as Buttazzo, Pratelli, Solimini and Stepanov \cite{BPSS}. It is a shape optimization problem in which one seeks an optimal street geometry for public transportation. The cost for a commuter to travel from one position to another (e.\,g.\ from home to work) is descibed using the \emph{urban metric}, the minimum travel distance based on the street network. The corresponding objective functional can be decomposed into the total cost for commuting (a Wasserstein-1 distance with respect to the urban metric) and for maintenance (proportional to the total length of the street network). The urban planning problem and urban metric were generalized in \cite{LSW}: As an additional parameter, a \emph{friction coefficient} was introduced which assigns a different street quality to different parts of the network. Both the urban metric and the maintenance cost in the generalized model depend on this friction coefficient.
	
	The \emph{branched transport problem} on the other hand is a non-convex and non-smooth variational problem on Radon measures in which one optimizes over mass fluxes from a given initial to a final mass distribution. A good introduction to classical branched transport is given in \cite{BCM}. 
	
	In \cite{LSW} it was shown that the generalized urban planning problem is equivalent to the branched transport problem with concave \emph{transportation cost}, extending a result from \cite{BW16} (equivalence between the classical urban planning problem and the branched transport problem with a specific transportation cost). 
%\todo[inline,color=green]{Description in brackets changed.}	
	The main step was to prove a Beckmann-type flux formulation of the Wasserstein-1 distance with respect to the generalized urban metric. To this end it was assumed that the generalized urban planning problem has finite cost, which allowed to derive a particular property of the network (namely, for any $C>0$ smaller than the friction outside the network, streets with friction smaller than $C$ have finite total length). This property was used to construct a minimizer for the Beckmann problem from a minimizer of the Wasserstein distance and vice versa.
	Without this assumption an explicit construction of minimizers is not possible since the Beckmann problem may not admit a minimizer. The main aim of the current work is to prove the Beckmann formulation using duality arguments (as recapitulated above for the Euclidean setting) without making use of the assumption. We prove it for the case of finite friction outside the network, but also show that this reformulation will stay true under mild assumptions (related to the network property from above) if the friction outside the network is infinite. We will first prove a Kantorovich\textendash Rubinstein formula for the Wasserstein distance (under these mild assumptions) in which the gradients of the Kantorovich potentials are bounded in terms of the friction coefficient and then apply Fenchel duality to obtain the Beckmann formulation.
	
	Apart from this result, we will provide a dual representation of the total network transportation cost induced by a mass flux from the branched transport problem (it is known that such a flux can be decomposed into a part concentrated on a network-like set and a diffuse part) under the assumption that the right derivative of the transportation cost is finite in zero. We will use this expression and the previously mentioned duality results to give a short alternative proof of the urban planning formulation of branched transport under the same assumption on the transportation cost.
	
	In the remainder of this section, we introduce the terminology used in urban planning and branched transport. Further, we summarize our main results and list the used notation and definitions. In \cref{subs6} we will derive the Kantorovich\textendash Rubinstein formula for the Wasserstein distance with respect to the generalized urban metric under mild assumptions on the street network. We will afterwards use this result to prove the Beckmann formulation under these mild assumptions or for the case of finite friction outside the network. 
	Finally, in \cref{uppeqbtp} we will introduce the friction coefficients as dual variables in a formula for the total network transportation cost of a mass flux from branched transport and give an alternative proof of the urban planning formulation of branched transport under the above-mentioned growth condition on the transportation cost.
	
	\subsection{Generalized urban planning}
	The original urban planning problem was introduced in \cite{BB,BPSS}. We will directly introduce the generalized setting and refer the reader to \cite[Section 1.2]{LSW} for a brief description of the relation between classical and generalized urban planning. We first introduce the set of admissible street networks.

	\begin{defin}[Street network, friction coefficient, city]
		\label{network}
		A \textbf{street network} is a pair $(S,b)$ with 
		\begin{itemize}
			\item $S\subset\R^n$ countably $1$-rectifiable and Borel measurable,
			\item $b:S\to[0,\infty)$ lower semicontinuous.
		\end{itemize}
		The function $b$ is called \textbf{friction coefficient}. A \textbf{city} is a triple $(S,a,b)$, where $(S,b)$ is a street network and $a\in[0,\infty]$ satisfies $b\leq a$ on $S$.
	\end{defin}
	The (friction) coefficients $b$ and $a$ describe the cost for travelling on and outside the network $S$ ($b$ may vary on $S$ since the street quality may vary). Note that, in general, offroad movement may be limited through buildings or the like. However, %we use the term ``city'' as an abbreviation, and because
	the results in this article will still hold true if we exclude finitely many sufficiently regular connected sets in $\R^n\backslash S$. The cost for travelling from $x$ to $y$ is desribed by the following (pseudo-)metric (in which $\Ha$ denotes the one-dimensional Hausdorff measure).
	\begin{defin}[Generalized urban metric]
		\label{gum}
		Let $(S,a,b)$ be a city. The associated \textbf{generalized urban metric} is defined as
		\begin{equation*}
		d_{S,a,b}(x,y)=\inf_{\gamma}\int_{\gamma([0,1])\cap S}b\,\mathrm{d}\Ha+a\Ha(\gamma([0,1])\setminus S),
		\end{equation*}
		where the infimum is taken over Lipschitz paths $\gamma:[0,1]\to\R^n$ with $\gamma(0)=x$ and $\gamma(1)=y$.
	\end{defin}
	Let $\mu_+$ and $\mu_-$ be probability measures on $\R^n$, which may for example describe the distribution of homes and workplaces. The total cost for the commuting population is then given by the Wasserstein distance with respect to $d_{S,a,b}$.
	\begin{defin}[Wasserstein distance, transport plans]
		\label{WdisTp}
		Let $(S,a,b)$ be a city. The \textbf{Wasserstein distance} between $\mu_+$ and $\mu_-$ with respect to $d_{S,a,b}$ is defined as
		\begin{equation*}
		W_{d_{S,a,b}}(\mu_+,\mu_-)=\inf_\pi\int_{\R^n\times\R^n}d_{S,a,b}\,\mathrm{d}\pi,
		\end{equation*}
		where the infimum is taken over all probability measures $\pi$ on $\R^n\times\R^n$ with $\pi(B\times\R^n)=\mu_+(B)$ and $\pi(\R^n\times B)=\mu_-(B)$ for all Borel sets $B\in\mathcal{B}(\R^n)$. Any such measure is called a \textbf{transport plan}. The set of transport plans is denoted by $\Pi(\mu_+,\mu_-)$.
	\end{defin}
	The Wasserstein distance is the first term which appears in the cost functional of the generalized urban planning problem. The second term corresponds to the total maintenance cost of the street network.
	\begin{defin}[Maintenance cost]
		A \textbf{maintenance cost} is a non-increasing function $c:[0,\infty)\to[0,\infty]$.
	\end{defin}
	The maintenance cost $c\circ b$ has the interpretation of a cost per length for maintaining a street with friction coefficient $b$.
	\begin{defin}[Generalized urban planning cost]\label{def:urbPlnCost}
		Given a maintenance cost $c$, set
		\begin{equation*}
		a=\inf c^{-1}(0).
		\end{equation*}
		The \textbf{generalized urban planning cost} of a city $(S,a,b)$ is given by
		\begin{equation*}
		\mathcal{U}^{c,\mu_+,\mu_-}[S,b]=W_{d_{S,a,b}}(\mu_+,\mu_-)+\int_Sc(b)\,\mathrm{d}\Ha,
		\end{equation*}
		and the \textbf{generalized urban planning problem} is the optimization problem
		\begin{equation*}
		\inf\left\{\mathcal{U}^{c,\mu_+,\mu_-}[S,b]\,\middle|\, (S,b)\text{ street network with }b\leq a\text{ on }S\right\}.
		\end{equation*}
	\end{defin}
	Note that $c\circ b$ is upper semi-continuous due to the assumptions on $b$ and $c$.
	%In this article, we will prove our main results under the assumption that the friction coefficient $a$ is finite (cf.\ \cref{assumpmetric,afinite,growthcond}). 
	
	\subsection{Generalized branched transport}
	To describe branched transport we use the Eulerian formulation due to Xia \cite{X}, which uses vector-valued Radon measures. Let $\mu_+$ and $\mu_-$ be probability measures on $\R^n$ that describe the initial an final distribution of mass (or of commuters in the context of urban planning). The cost to move an amount of mass per unit distance is given by the following function.
	\begin{defin}[Transportation cost]
		A \textbf{transportation cost} is a non-decreasing concave function $\tau:[0,\infty)\to[0,\infty)$ with $\tau(0)=0$.
	\end{defin}
	A classical example for a transportation cost is $\tau(m)=m^\alpha$ for some $\alpha\in(0,1)$.
	The concavity of $\tau$ encodes that the transport gets cheaper per mass particle if more mass is transported together.
	%For transport along a finite graph the cost is readily defined as follows.
	To describe transport along a finite graph we define polyhedral mass fluxes (or equivalently, finite oriented weighted graphs) between distributions $\mu_+$ and $\mu_-$ which are concentrated in finitely many points.
	\begin{defin}[Polyhedral mass flux and branched transport cost]
		Assume that $\mu_+$ and $\mu_-$ are finite sums of weighted Dirac measures, i.\,e.,
		\begin{equation*}
		\mu_+=\sum_{i=1}^{M}f_i\delta_{x_i}\text{\qquad and\qquad}\mu_-=\sum_{j=1}^{N}g_j\delta_{y_j},
		\end{equation*}
		where $f_i,g_j\in[0,1]$ satisfy $\sum_if_i=\sum_jg_j=1$ and $x_i,y_j\in\R^n$. A \textbf{polyhedral mass flux} between $\mu_+$ and $\mu_-$ is a vector-valued Radon measure $\F\in\mathcal{M}^n(\R^n)$ which satisfies $\textup{div}(\F)=\mu_+-\mu_-$ in the distributional sense and can be written as
		\begin{equation*}
		\F=\sum_em_e\vec{e}\Ha\mres e,
		\end{equation*}
		where the sum is over finitely many edges $e=x_e+[0,1](y_e-x_e)\subset\R^n$ with orientation $\vec{e}=(y_e-x_e)/|y_e-x_e|$, the coefficients $m_e$ are positive real weights, and $\Ha\mres e$ is the restriction of the one-dimensional Hausdorff measure to $e$. The \textbf{branched transport cost} of $\F$ with respect to a transportation cost $\tau$ is defined as
		\begin{equation*}
		\mathcal{J}^{\tau,\mu_+,\mu_-}[\F]=\sum_e\tau(m_e)\Ha(e).
		\end{equation*}
	\end{defin}
	We use the lower semi-continuous envelope of $\mathcal{J}^{\tau,\mu_+,\mu_-}$ to get the branched transport cost of a general vector-valued Radon measure with distributional divergence equal to $\mu_+-\mu_-$.
	\begin{defin}[Mass flux and branched transport cost]
		\label{defags}
		A vector-valued Radon measure $\F\in\mathcal{M}^n(\R^n)$ is called \textbf{mass flux} between the probability measures $\mu_+$ and $\mu_-$ if there exist two sequences of probability measures $\mu_+^k,\mu_-^k$ and a sequence of polyhedral mass fluxes $\F_k$ with $\textup{div}(\F_k)=\mu_+^k-\mu_-^k$ such that $\F_k\xrightharpoonup{*}\F$ and $\mu_{\pm}^k\xrightharpoonup{*}\mu_{\pm}$, where $\xrightharpoonup{*}$ indicates the weak-$*$ convergence in duality with continuous functions. We write $(\F_k,\mu_+^k,\mu_-^k)\xrightharpoonup{*}(\F,\mu_+,\mu_-)$. If $\F$ is a mass flux, then the \textbf{branched transport cost} of $\F$ is defined as
		\begin{equation*}
		\mathcal{J}^{\tau,\mu_+,\mu_-}[\F]=\inf\left\{\liminf_k\mathcal{J}^{\tau,\mu_+^k,\mu_-^k}[\F_k]\,\bigg|\, (\F_k,\mu_+^k,\mu_-^k)\xrightharpoonup{*}(\F,\mu_+,\mu_-)\right\}.
		\end{equation*}
		The corresponding \textbf{branched transport problem} is the optimization problem
		\begin{equation*}
		\inf\{ \mathcal{J}^{\tau,\mu_+,\mu_-}[\F]\,|\,\F\in\mathcal{M}^n(\R^n),\textup{div}(\F)=\mu_+-\mu_- \}.
		\end{equation*}
	\end{defin}

	\subsection{Summary of results}
	Our main results are
	\begin{itemize}
		\item
		a Kantorovich\textendash Rubinstein formula for the Wasserstein distance $W_{d_{S,a,b}}$ under mild assumptions on the city $(S,a,b)$ (\cref{KRformula}),
		\item
		a Beckmann formula for $W_{d_{S,a,b}}$ under the same assumptions on the city or $a<\infty$ (\cref{BWfinal}),
		\item a dual formula for a total network transportation cost under the growth condition $\tau'(0)<\infty$ (\cref{subst}), and
		\item a new proof of the equivalence between urban planning and branched transport for the case $\tau'(0)<\infty$ using our duality results (\cref{finalthm}).
	\end{itemize}
	Let $\mu_+,\mu_-$ be probability measures on $\R^n$ with bounded supports, contained in $\mathcal{C}=[-1,1]^n$ without loss of generality. We will prove our first main result under the following assumption.
	\begin{assump}[Regularity of the city $(S,a,b)$]
		\label{assumpmetric}
		The friction coefficient $b$ is bounded away from $0$, and the function which maps $x\in\R^n$ to the friction coefficient at $x$, i.e. $b(x)$ if $x\in S$, $a$ else, is lower semi-continuous.
	\end{assump}
	The last condition in \cref{assumpmetric} is for instance automatically satisfied if the set $S$ of streets is closed.
	Note, though, that it is strictly weaker than requiring closedness of $S$ (which typically is not satisfied for branched transport networks, see \cref{exbtn}).
	\begin{thm}[Version of the Kantorovich\textendash Rubinstein formula]
		\label{KRformula}
		Let $(S,a,b)$ be a city such that \cref{assumpmetric} is satisfied. Then the Wasserstein-1 distance between $\mu_+$ and $\mu_-$ is given by
		\begin{equation*}
		W_{d_{S,a,b}}(\mu_+,\mu_-)=\sup_{\varphi}\int_{\mathcal{C} }\varphi\,\mathrm{d}(\mu_+-\mu_-),
		\end{equation*}
		where the supremum is taken over functions $\varphi\in C^1(\mathcal{C})$ with $|\nabla\varphi|\leq b$ on $S$ and $|\nabla\varphi|\leq a$ in $\mathcal{C}\backslash S$.
	\end{thm}
	The work here is to show that one may restrict the supremum to differentiable functions with the given gradient constraint;
	for lower regularity requirements on $\varphi$ it is a standard result.
	\begin{assump}[Finite friction outside network]
		\label{afinite}
		The cost $a$ for travelling outside the network is finite.
	\end{assump}
	
	The Beckmann formulation of the Wasserstein distance will actually be true under \cref{assumpmetric} \textit{or} \cref{afinite}.
	\begin{thm}[Beckmann-type formulation of $W_{d_{S,a,b}}(\mu_+,\mu_-)$]
		\label{BWfinal}
		Let $(S,a,b)$ be a city and suppose that \cref{assumpmetric} or \cref{afinite} is satisfied. Then we have
		\begin{equation*}
		W_{d_{S,a,b}}(\mu_+,\mu_-)=B_{S,a,b}(\mu_+,\mu_-)
		\qquad\text{for}\qquad
		B_{S,a,b}(\mu_+,\mu_-)=\inf_{\xi,\F^\perp}\int_S b|\xi|\,\mathrm{d}\Ha+a|\F^\perp|(\mathcal{\mathcal{C}}),
		\end{equation*}
		where the infimum is taken over $\xi\in L^1(\Ha\mres S;\R^n)$ and $\F^\perp\in\mathcal{M}^n(\mathcal{C})$ with $\F^\perp\mres S=0$ and $\textup{div}(\xi\Ha\mres S+\F^\perp)=\mu_+-\mu_-$.
	\end{thm}
	It is unclear whether the result even stays true if one drops all assumptions
	-- we expect this to be the case, but the method of proof via duality does not allow to draw such conclusions. A brief summary of the Beckmann formulation of the Wasserstein distance with respect to $d_{S,a,b}$ is given in \cref{fig:BWfigure}.
	\begin{figure}
		\centering
			\begin{tabular}{|cx{1.5pt}  c | c|} 
 \hline
 $a$ & current article & \cite{LSW} \\ 
 \hline\hlinewd{1.5pt}
 finite & no assumption & \cite[Asm.\,1.3.1]{LSW} \\ 
 \hline
 infinite & \cref{assumpmetric} & \cite[Asm.\,1.3.1]{LSW} \\
 \hline
\end{tabular}
		\caption{Summary of the cases for which the Beckmann formulation of the Wasserstein distance with respect to $d_{S,a,b}$ has been proven. The requirement in \cite[Asm.\,1.3.1]{LSW} is that for every $C\in[0,a)$ we have $\mathcal{H}^1(\{ b\leq C\})<\infty$. In all four cases a minimizer for $W_{d_{S,a,b}}(\mu_+,\mu_-)$ exists (if the Wasserstein distance is finite). The same holds true for the Beckmann formulation except for $a\in(0,\infty]$ with no additional condition (see \cite[Exm.\,1.3.1]{LSW}).}
		\label{fig:BWfigure}
	\end{figure}
	The analogue of \cref{afinite} in the context of branched transport turns out to be the following.
	\begin{assump}[Growth condition]
		\label{growthcond}
		The average transportation cost per unit mass $m\mapsto\tau(m)/m$ is bounded, i.\,e., $\tau'(0)=\lim_{m\searrow 0}\tau(m)/m<\infty$.
	\end{assump}
	The maintenance cost corresponding to a transportation cost for the branched transport problem is defined via convex conjugation.
	\begin{defin}[Maintenance cost associated with $\tau$]
		Let $\tau:[0,\infty)\to[0,\infty)$ be a transportation cost. We extend $\tau$ to a function on $\R$ via $\tau(m)=-\infty$ for all $m<0$. The associated \textbf{maintenance cost} is defined by $\varepsilon(b)=(-\tau)^*(-b)=\sup_{m\geq 0}\tau(m)-bm$ for any $b\in[0,\infty)$.
	\end{defin}
	It can be shown that the branched transport cost of a mass flux $\F=\xi\Ha\mres S+\F^\perp$ with divergence equal to $\mu_+-\mu_-$ can be decomposed into the total cost for transportation on the network $S$ and a cost term for the diffuse part $\F^\perp$ (see \cite[Prop.\,2.32]{BW} or \cite[Lemma\,3.1.8]{LSW}). The next formula shows that the network cost term can be seen as an optimization problem over friction coefficients.
	\begin{thm}[Dual formula for total network transportation cost]
		\label{subst}
		Assume that $S\subset\mathcal{C}$ is countably $1$-rectifiable and Borel measurable. Furthermore, let $\xi\in L^1(\mathcal{H}^1\mres S;\R^n)$ represent a mass flux on $S$ and suppose that $\tau$ is a transportation cost which satisfies \cref{growthcond}. Then we have
		\begin{equation*}
		\int_S\tau (|\xi |)\,\mathrm{d}\mathcal{H}^1=\inf_b\int_Sb|\xi |\,\mathrm{d}\mathcal{H}^1+\int_S\varepsilon (b)\,\mathrm{d}\mathcal{H}^1,
		\end{equation*}
		where the infimum is taken over lower semi-continuous functions $b:S\to[0,\tau'(0)]$.
	\end{thm}
	At the end of this article, we will use the previous duality results to provide an alternative proof of the urban planning formulation of branched transport under \cref{growthcond}.
	\begin{thm}[{{Bilevel formulation of the branched transport problem \cite[Theorem 1.3.4]{LSW}}}]
		\label{finalthm}
		Let $\tau$ be a transportation cost with associated maintenance cost $\varepsilon$. Then the branched transport problem can equivalently be written as urban planning problem,
		\begin{equation*}
		\inf_{\F}\mathcal{J}^{\tau,\mu_+,\mu_-}[\F]=\inf_{S,b}\mathcal{U}^{\varepsilon,\mu_+,\mu_-}[S,b],
		\end{equation*}
		where the infima are taken over $\F\in\mathcal{M}^n(\mathcal{C})$ with $\textup{div}(\F)=\mu_+-\mu_-$ and street networks $(S,b)$ with $S\subset\mathcal{C}$ and $b\leq a=\tau'(0)$ on $S$.
	\end{thm}
	
	\Cref{BWfinal,finalthm,subst} reveal the convex duality structure that connects branched transport with urban planning:
	Essentially, both problems can be (formally) written as
	\begin{equation*}
	\inf_{S,b,\xi,\mathcal{F}^\perp}\int_S b|\xi|\,\mathrm d\Ha+a|\F^\perp|(\mathcal{\mathcal{C}})+\int_S\varepsilon(b)\,\mathrm d\Ha+\iota_{\{ \mu_+-\mu_- \}}(\textup{div}(\xi\Ha\mres S+\F^\perp))
	\end{equation*}
	with $a=\sup_S b$, which is separately convex in $(\xi,\mathcal F^\perp)$ and in $b$.
	Eliminating $(\xi,\mathcal F^\perp)$ (by minimizing in those variables) leads to urban planning, eliminating $(b,a)$ to branched transport.
	The nonconvexity of the problem arises from the bilinearity in which $(|\xi|,|\mathcal F^\perp|)$ and $(b,a)$ are coupled.
	
	To the end of this section, we briefly discuss \cref{assumpmetric} and \cref{afinite} with regard to \cref{KRformula}. The conditions in \cref{KRformula} are sharp (cf.~\cref{sharp1} and \cref{sharp2}). We believe that the condition that $b$ is bounded away from zero in \cref{KRformula} can be dropped under \cref{afinite}. We use this condition in the proof of \cref{papprox} to get uniformly bounded lenghts of certain almost optimal paths with respect to $d_{S,a,b}$. It is possible to relax this condition to the requirement that there exists some constant $C>0$ such that $\mathcal{H}^1(\{ b\leq C\})<\infty$ (an implication of \cite[Asm.\,1.3.1]{LSW}), but this would make the arguments more technical. Moreover, we will use the boundedness away from zero in the proof of \cref{BWfinal} under \cref{assumpmetric}. For the sake of completeness, we finally give an example which shows that it is reasonable not to restrict to closed networks $S$ in the context of branched transport. The requirement that $b$, extended with value $a$, is lower semi-continuous in $\R^n$ from \cref{assumpmetric} turns out to be a natural weakening of the requirement that $S$ is closed.
	\begin{figure}
		\centering
		\begin{subfigure}[b]{0.3\textwidth}
			\centering
			\begin{tikzpicture}
			\coordinate (x1) at (0,0);
			\coordinate (x2) at (0,1);
			\coordinate (x3) at (1.2,-0.2);	
			
			\coordinate (y1) at (2,2.1);
			\coordinate (y2) at (3,1.5);
			\coordinate (y3) at (3.5,0.3);
			\coordinate (y4) at (2.8,-0.3);
			\coordinate (y5) at (0.5,-2);
			
			\coordinate (v1) at (0.4,-1);
			\coordinate (v2) at (0.5,0.5);
			\coordinate (v3) at (1.2,0.7);
			\coordinate (v4) at (1.8,1.3);
			\coordinate (v5) at (2.1,1.2);
			\coordinate (v6) at (2.9,0.4);	
			
			\node[circle,fill=black,inner sep=0.5pt,minimum size=0.1cm] at (x1) {};
			\node[circle,fill=black,inner sep=0.5pt,minimum size=0.1cm] at (x2) {};
			\node[circle,fill=black,inner sep=0.5pt,minimum size=0.1cm] at (x3) {};
			
			\node[draw=black,circle,fill=white,inner sep=0.5pt,minimum size=0.1cm] at (y1) {};
			\node[draw=black,circle,fill=white,inner sep=0.5pt,minimum size=0.1cm] at (y2) {};
			\node[draw=black,circle,fill=white,inner sep=0.5pt,minimum size=0.1cm] at (y3) {};
			\node[draw=black,circle,fill=white,inner sep=0.5pt,minimum size=0.1cm] at (y4) {};
			\node[draw=black,circle,fill=white,inner sep=0.5pt,minimum size=0.1cm] at (y5) {};			
			
			\draw[-stealth] (x1) -- (v1);
			\draw[stealth-] (v1) -- (x3);
			\draw[-stealth] (v1) -- (y5);	
			\draw[-stealth] (x1) -- (v2);
			\draw[stealth-] (v2) -- (x2);
			\draw[-stealth] (v2) -- (v3);
			\draw[stealth-] (v3) -- (x3);
			\draw[-stealth] (v3) -- (v4);
			\draw[-stealth] (v4) -- (v5);
			\draw[-stealth] (v4) -- (y1);
			\draw[stealth-] (y2) -- (v5);
			\draw[-stealth] (v5) -- (v6);
			\draw[stealth-] (y3) -- (v6);
			\draw[-stealth] (v6) -- (y4);
			\end{tikzpicture}
			\caption{Polyhedral mass flux.}
			\label{sf:fig2a}
		\end{subfigure}
		\begin{subfigure}[b]{0.3\textwidth}
			\centering
			\begin{tikzpicture}
			\foreach \x in {1,2,...,29}
			\foreach \y in {1,2,...,29}
			{		
				\node[circle,fill=black,inner sep=0.2pt,minimum size=0.01cm] at ({2*\x/30},{2*\y/30}) {};
				\node[circle,fill=black,inner sep=0.2pt,minimum size=0.01cm] at ({1.5+2*\x/30},{2.3+2*\y/30}) {};
				\draw[line width = 0.0 mm] ({2*\x/30},{2*\y/30}) -- ({1.5+2*\x/30},{2.3+2*\y/30});
			}
			
			\draw[densely dashed,white,line width = 0.0] (1.5,4.3) -- (1.5,2.3) -- (3.5,2.3);
			\draw[-stealth,lightgray,thick] (1,1) -- (2.5,3.3);
			\draw[densely dashed, white, line width = 0.0] (0,2) -- (2,2) -- (2,0);
			
			\node[label={[label distance=-0.2cm,text=lightgray]0:  $v$}] at (2.15,2.55) {};
			
			\end{tikzpicture}	
			\caption{Network not closed.}
			\label{sf:fig2b}
		\end{subfigure}
		\begin{subfigure}[b]{0.3\textwidth}
			\centering
			\begin{tikzpicture}
			\coordinate (x) at (0,0);
			
			\coordinate (y1) at (1,0);
			\coordinate (y2) at (-1,0);
			\coordinate (y3) at (0,1);
			\coordinate (y4) at (0,-1);
			
			\coordinate (y5) at (1.5,0.5);
			\coordinate (y6) at (1.5,-0.5);
			\coordinate (y7) at (-1.5,0.5);
			\coordinate (y8) at (-1.5,-0.5);
			\coordinate (y9) at (0.5,1.5);
			\coordinate (y10) at (-0.5,1.5);
			\coordinate (y11) at (0.5,-1.5);
			\coordinate (y12) at (-0.5,-1.5);	
			
			\coordinate (y13) at (1.5,1);
			\coordinate (y14) at (1.5,-1);
			\coordinate (y15) at (2,0.5);
			\coordinate (y16) at (2,-0.5);
			\coordinate (y17) at (-1.5,1);
			\coordinate (y18) at (-1.5,-1);
			\coordinate (y19) at (-2,0.5);
			\coordinate (y20) at (-2,-0.5);
			\coordinate (y21) at (0.5,2);
			\coordinate (y22) at (-0.5,2);
			\coordinate (y23) at (1,1.5);
			\coordinate (y24) at (-1,1.5);
			\coordinate (y25) at (0.5,-2);
			\coordinate (y26) at (-0.5,-2);
			\coordinate (y27) at (1,-1.5);
			\coordinate (y28) at (-1,-1.5);
			
			\draw[white,fill=lightgray] (0,0) circle (2.25);
			\draw[white,fill=white] (0,0) circle (0.75);
			
			\node[circle,fill=black,inner sep=0.5pt,minimum size=0.1cm] at (x) {};
			
			\node[draw=black,circle,fill=white,inner sep=0.5pt,minimum size=0.1cm] at (y1) {};
			\node[draw=black,circle,fill=white,inner sep=0.5pt,minimum size=0.1cm] at (y2) {};
			\node[draw=black,circle,fill=white,inner sep=0.5pt,minimum size=0.1cm] at (y3) {};
			\node[draw=black,circle,fill=white,inner sep=0.5pt,minimum size=0.1cm] at (y4) {};
			
			\node[draw=black,circle,fill=white,inner sep=0.5pt,minimum size=0.1cm] at (y13) {};
			\node[draw=black,circle,fill=white,inner sep=0.5pt,minimum size=0.1cm] at (y14) {};
			\node[draw=black,circle,fill=white,inner sep=0.5pt,minimum size=0.1cm] at (y15) {};
			\node[draw=black,circle,fill=white,inner sep=0.5pt,minimum size=0.1cm] at (y16) {};
			\node[draw=black,circle,fill=white,inner sep=0.5pt,minimum size=0.1cm] at (y17) {};	
			\node[draw=black,circle,fill=white,inner sep=0.5pt,minimum size=0.1cm] at (y18) {};
			\node[draw=black,circle,fill=white,inner sep=0.5pt,minimum size=0.1cm] at (y19) {};
			\node[draw=black,circle,fill=white,inner sep=0.5pt,minimum size=0.1cm] at (y20) {};
			\node[draw=black,circle,fill=white,inner sep=0.5pt,minimum size=0.1cm] at (y21) {};
			\node[draw=black,circle,fill=white,inner sep=0.5pt,minimum size=0.1cm] at (y22) {};	
			\node[draw=black,circle,fill=white,inner sep=0.5pt,minimum size=0.1cm] at (y23) {};
			\node[draw=black,circle,fill=white,inner sep=0.5pt,minimum size=0.1cm] at (y24) {};
			\node[draw=black,circle,fill=white,inner sep=0.5pt,minimum size=0.1cm] at (y25) {};	
			\node[draw=black,circle,fill=white,inner sep=0.5pt,minimum size=0.1cm] at (y26) {};
			\node[draw=black,circle,fill=white,inner sep=0.5pt,minimum size=0.1cm] at (y27) {};
			\node[draw=black,circle,fill=white,inner sep=0.5pt,minimum size=0.1cm] at (y28) {};
			
			\draw[-stealth] (y5) -- (y13);
			\draw[-stealth] (y5) -- (y15);
			\draw[-stealth] (y6) -- (y14);
			\draw[-stealth] (y6) -- (y16);
			
			\draw[-stealth] (y7) -- (y17);
			\draw[-stealth] (y7) -- (y19);
			\draw[-stealth] (y8) -- (y18);
			\draw[-stealth] (y8) -- (y20);
			
			\draw[-stealth] (y9) -- (y21);
			\draw[-stealth] (y9) -- (y23);
			\draw[-stealth] (y10) -- (y22);
			\draw[-stealth] (y10) -- (y24);
			
			\draw[-stealth] (y11) -- (y25);
			\draw[-stealth] (y11) -- (y27);
			\draw[-stealth] (y12) -- (y26);
			\draw[-stealth] (y12) -- (y28);
			
			\node[draw=black,circle,fill=white,inner sep=0.5pt,minimum size=0.1cm] at (y5) {};	
			\node[draw=black,circle,fill=white,inner sep=0.5pt,minimum size=0.1cm] at (y6) {};
			\node[draw=black,circle,fill=white,inner sep=0.5pt,minimum size=0.1cm] at (y7) {};
			\node[draw=black,circle,fill=white,inner sep=0.5pt,minimum size=0.1cm] at (y8) {};
			\node[draw=black,circle,fill=white,inner sep=0.5pt,minimum size=0.1cm] at (y9) {};
			\node[draw=black,circle,fill=white,inner sep=0.5pt,minimum size=0.1cm] at (y10) {};	
			\node[draw=black,circle,fill=white,inner sep=0.5pt,minimum size=0.1cm] at (y11) {};
			\node[draw=black,circle,fill=white,inner sep=0.5pt,minimum size=0.1cm] at (y12) {};
			
			\draw[-stealth] (y1) -- (y5);
			\draw[-stealth] (y1) -- (y6);
			\draw[-stealth] (y2) -- (y7);
			\draw[-stealth] (y2) -- (y8);
			\draw[-stealth] (y3) -- (y9);
			\draw[-stealth] (y3) -- (y10);
			\draw[-stealth] (y4) -- (y11);
			\draw[-stealth] (y4) -- (y12);
			
			\node[draw=black,circle,fill=white,inner sep=0.5pt,minimum size=0.1cm] at (y1) {};
			\node[draw=black,circle,fill=white,inner sep=0.5pt,minimum size=0.1cm] at (y2) {};
			\node[draw=black,circle,fill=white,inner sep=0.5pt,minimum size=0.1cm] at (y3) {};
			\node[draw=black,circle,fill=white,inner sep=0.5pt,minimum size=0.1cm] at (y4) {};
			
			\draw[-stealth] (x) -- (y1);
			\draw[-stealth] (x) -- (y2);
			\draw[-stealth] (x) -- (y3);
			\draw[-stealth] (x) -- (y4);
			
			\node[label={[label distance=-0.2cm]0: \tiny $\alpha 4^{1-\alpha}$}] at (0,0.5) {};
			\node[label={[label distance=-0.2cm]0: \tiny $\alpha 8^{1-\alpha}$}] at (0.25,1.25) {};
			\node[label={[label distance= -0.2cm]90: \tiny $\alpha 16^{1-\alpha}$}] at (0.98,1.53) {};
			\end{tikzpicture}
			\caption{Approximate irrigation of a ring.}
			\label{sf:fig2c}
		\end{subfigure}
		\caption{Sketches for \cref{exbtn}.}
	\end{figure}
	\begin{examp}[Examples of branched transport networks]
	\label{exbtn}
	\begin{enumerate}[label=(\alph*)]
	\item
	Clearly, any polyhedral mass flux induces a closed network $S$ of finite length (\cref{sf:fig2a}). If $b$ is lower semi-continuous on $S$ with $b\leq a\in[0,\infty]$, then the extension of $b$ with value $a$ to $\R^n$ is automatically lower semi-continuous so that \cref{assumpmetric} is fulfilled as soon as $b$ is bounded away from zero.
	\item
	Recall from \cref{finalthm} that, in the context of branched transport, the friction coefficient $a$ equals $\tau'(0)$, where $\tau$ is some transportation cost of the branched transport problem. For finite $a$, e.\,g.\ for $\tau(m)=m$ (Wasserstein transport), the network $S$ induced by an optimal mass flux is not necessarily closed: Let $\mu_+=\sum_im_i\delta_{x_i}$ with $m_i\in(0,1]$ be a probability measure concentrated on pairwise disjoint points $x_i\in\R^n$ such that $\bigcup_i\{ x_i\}$ is not closed (e.\,g.\ the point cloud $(\mathbb{Q}\cap (0,1))^2$ in \cref{sf:fig2b}). Define $\mu_-(B)=\mu_+(B-v)$ for some fixed $v\in\R^n\backslash\{ 0 \}$. Then an optimal Wasserstein mass flux is concentrated on the set $S=\bigcup_i[x_i,x_i+v]$, which is not closed. Nevertheless, the induced city satisfies \cref{assumpmetric}: From \cref{subst} we see directly that $b\equiv 1$ on $S$ is the right choice so that the extension to $\R^n$ with value $a=\tau'(0)=1$ is lower semi-continuous.
	\item
	For an example with $a=\infty$ we can choose $\mu_+=\delta_0$ and $\mu_-=\mathcal{L}^n\mres U$ the Lebesgue measure on some open and bounded set $U\subset\R^n$ with $\Le^n(U)=1$ (in \cref{sf:fig2c} given by a ring). It can then be shown that the branched transport problem with classical transportation cost $\tau(m)=m^\alpha$ admits a minimizer with finite energy if $\alpha\in (1-1/n,1)$ \cite[Thm.\,3.1]{X}, where any optimal mass flux is of the form $\mathcal{F}=\xi\Ha\mres S$ with $S$ countably $1$-rectifiable (which we will later see from the formula in \cref{version}, in which the diffuse part $\mathcal{F}^\perp$ must vanish). If $S$ was closed, then we would get $U\subset S$ which cannot be true for $n>1$ by $\mathcal{L}^n(S)=0$ -- rather $S$ will be an infinitely refining, fractal network whose first iterations are illustrated in \cref{sf:fig2c} and whose mass flux becomes the smaller the finer the branches are. From \cref{subst} one can already guess that the corresponding $b$ will tend to infinity the smaller the flux and branches are (its values along one path are indicated in \cref{sf:fig2c}; compare with the formula in \cref{constrmin}. Since $a=\tau'(0)=\infty$, \cref{assumpmetric} is fulfilled.% It is now possible to remove the divergence-free part of $\mathcal{F}$ and to define a lower semi-continuous function $b$ on the (adjusted) network $S$ such that the resulting street network $(S,b)$ is optimal for the urban planning problem (see \cite[Sect.\,1.3]{LSW}). For finer and optimal approximations with $N$ sinks the value of $b$ on branches close to $U$ will tend to infinity as $N\to\infty$.
	\end{enumerate}
	\end{examp}
	\subsection{Notation and definitions}
	\label{notdefs}
	Throughout the article we will use the following notation and definitions.
	\begin{itemize}
		\item $I= [0,1]$ denotes the \emph{unit interval}. We will use this notation if $I$ represents the domain of a path.
		\item $\mathcal{S}^{n-1}$ denotes the \emph{unit sphere}.
		\item $\mathcal{C}$ denotes the \emph{hypercube} $[-1,1]^n$.
		\item We write $B_r(x)$ for the \emph{open Euclidean ball} with radius $r>0$ and centre $x\in\R^n$ and $\overline{B}_r(x)$ for its closure.
		\item $\mathcal{L}^k$ denotes the $k$-dimensional \emph{Lebesgue measure}. We write $\mathcal{L}=\mathcal{L}^1$.
		\item $\mathcal{H}^k$ indicates the $k$-dimensional \emph{Hausdorff measure}.
		\item Let $V$ be a normed vector space. We write $V^*$ for the \emph{topological dual space}. For $v\in V$ and $v^*\in V^*$ we denote the pairing between $v$ and $v^*$ by $\langle v,v^*\rangle$.
		\item Let $A$ be a topological space. We write $\mathcal{B}(A)$ for the \emph{$\sigma$-algebra of Borel subsets} of $A$.
		\item $C(\mathcal{C})$ denotes the space of real-valued continuous functions on $\mathcal{C}$. The subspace of Lipschitz functions is denoted $C^{0,1}(\mathcal{C})$. Further, we write $C^1(\mathcal{C})$ for the space of functions $f\in C(\mathcal{C})$ which are continuously differentiable on $(-1,1)^n$ such that $f'$ can be continued to an element of $C(\mathcal{C})$.  
		\item Assume that $(\Omega,\mathcal{A},\mu)$ is a measure space. We write $L^1(\mu;\R^n)$ for the \emph{Lebesgue space} of equivalence classes of $\mathcal{A}$-$\mathcal{B}(\R^n)$-measurable functions $f:\Omega\to\R^n$ with $\int_\Omega|f|\,\mathrm{d}\mu<\infty$, where two such functions belong to the same class if they coincide $\mu$-almost everywhere. For $\sigma$-finite $\mu$ this definition corresponds to the quotient of the Lebesgue space $L_1(\mu,\R^n)$ defined in \cite[\S\,2.4.12]{F} by the subspace $\{ f\,|\, f=0\, \mu\text{-almost everywhere} \}$.
		\item Let $\mu:\mathcal{A}\to X$ be a map from a $\sigma$-algebra $\mathcal{A}$ to some set $X$ (e.\,g., a scalar- or vector-valued measure). For any $A\in\mathcal{A}$ we define the \emph{restriction} $\mu\mres A:\mathcal{A}\to X$ of $\mu$ to $A$ by
		\begin{equation*}
		(\mu\mres A) (B)=\mu(A\cap B)
		\quad\text{for all }B\in\mathcal{A}.
		\end{equation*}
		\item A set $S\subset\R^n$ is said to be \emph{countably $k$-rectifiable} (following \cite[p.\,251]{F}) if it is the countable union of $k$-rectifiable sets. More precisely, 
		\begin{equation*}
		S=\bigcup_{i=1}^\infty f_i(A_i),
		\end{equation*}
		where $A_i\subset\R^k$ is bounded and $f_i:A_i\to\R^n$ Lipschitz continuous. If $S$ is countably $k$-rectifiable and $\mathcal{H}^k$-measurable, then we can apply \cite[Lem.\,3.2.18]{F} which yields the existence of bi-Lipschitz functions $g_i:C_i\to S$ with $C_i\subset\R^k$ compact, $T_i=g_i(C_i)$ pairwise disjoint and 
		\begin{equation*}
		S=T_0\cup \bigcup_{i=1}^\infty T_i
		\end{equation*}
		with $\mathcal{H}^k(T_0)=0$. The sequence 
		\begin{equation*}
		S^N=\bigcup_{i=1}^NT_i
		\end{equation*}
		will be called an \emph{approximating sequence for} $S$.
		%\todo[inline]{Can a $k$-rectifiable set indeed also be $k-1$-dimensional?}
		%\todo[inline,color=green]{Yes: by Federer's definition a line segment in $\R^3$ is $2$-rectifiable.}
		\item $\mathcal{M}^k(A)=\{\F:\mathcal{B}(A)\to\R^k\,\sigma\text{-additive}\}$ denotes the set of $\R^k$-valued \emph{Radon measures} on a Polish space $A$. Note that every $\F\in\mathcal{M}^k(A)$ is automatically regular and of bounded variation (cf.\ \cite[p.\,343]{Els} and \cite[XI, 4.5., Thm.\,8]{L}). More specifically, the \emph{total variation measure} $|\F|$ is regular and satisfies $|\F|(A)<\infty$. We indicate the weak-$*$ convergence of Radon measures by $\ws$. The measure $\F\in\mathcal{M}^k(A)$ is called $\mathcal{H}^l$\emph{-diffuse} if $\F(B)=0$ for all $B\in\mathcal{B}(A)$ with $\mathcal{H}^l(B)<\infty$ \cite[p.\,2]{Sil}.
		\item For any closed subset $A\subset\R^n$ we write $\mathcal{DM}^n(A)=\{\F\in\mathcal{M}^n(A)\,|\,\textup{div}(\F)\in\mathcal{M}^1(A) \}$, where $\textup{div}$ denotes the distributional divergence. These vector-valued Radon measures were termed \emph{divergence measure vector fields} in \cite[p.\,2]{Sil}.
		%		\item $\Theta^{*k}(\mu,.)$ denotes the \emph{upper $k$-dimensional density} of a Radon measure  $\mu:\mathcal{B}(\R^n)\to[0,\infty)$ \cite[p.\,13]{LS}. It is for every $x\in\R^n$ given by
		%		\begin{equation*}
		%		\Theta^{*k}(\mu,x)=\limsup_{r\searrow 0}\frac{\mu(B_r(x))}{r^k\omega_k},
		%		\end{equation*}
		%		where $\omega_k$ denotes the volume of the $k$-dimensional unit ball.
		%		\item The \emph{pushforward} $f_\#\mu$ of a measure $\mu$ on $X$ under a measurable map $f : X \to Y$ is the measure defined by $f_\#\mu (A)=\mu (f^{-1}(A))$ for all measurable subsets $A\subset Y$.
		%		\item $p_i:A_1\times\ldots\times A_k\to A_i$ abbreviates the \emph{projection} on the $i$-th component.
		\item We write the \emph{arc length} of a Lipschitz path $\gamma:I\to\mathcal{C}$ as $\text{len}(\gamma)=\int_{I}|\dot{\gamma}|\,\mathrm{d}\mathcal{L}$.
				\item We write $\Gamma^{xy}=\{f:I\to\mathcal{C}\textup{ Lipschitz}\,|\,f(0)=x,f(1)=y\}$ for $x,y\in\mathcal{C}$. %Further, for $x,y\in\mathcal{C}$ and $C>0$ let 
		%		\begin{equation*}
		%		\Gamma_C=\{ f\in\Gamma\,|\,\text{Lip}(f)\leq C\}\text{\qquad and\qquad}\Gamma_C^{xy}=\{f\in\Gamma_C\,|\,f(0)=x,f(1)=y\}.
		%		\end{equation*}
		%		\item For any Lipschitz path $\gamma:[t_1,t_2]\to\mathcal{C}$ we write $\text{md}(\gamma,t_0)$ for the \emph{metric differential} of $\gamma$ at $t_0\in (t_1,t_2)$ \cite[p.\,115]{Kirch}, which can be applied to $u\in\R$ by
		%		\begin{equation*}
		%		\text{md}(\gamma,t_0)(u)=\lim_{h\searrow 0}\frac{|\gamma(t_0+hu)-\gamma(t_0)|}{h}
		%		\end{equation*}
		%		if the limit exists. Further, the \emph{metric derivative} \cite[p.\,24]{AGS} of $\gamma$ at $t_0$ is given by 
		%		\begin{equation*}
		%		|\gamma'|(t_0)=\lim_{h\to 0}\frac{|\gamma(t_0+h)-\gamma(t_0)|}{|h|}
		%		\end{equation*}
		%		if this limit exists. Note that by Rademacher's theorem $\dot{\gamma}(t_0)$ exists for $\mathcal{L}$-almost all $t_0$, and for those $t_0$ we have
		%		\begin{equation*}
		%		|\dot{\gamma}(t_0)|=|\gamma'|(t_0)=\text{md}(\gamma,t_0)(1).
		%		\end{equation*}
		\item We will identify the image of a Lipschitz path $\gamma:I\to\mathcal{C}$ with its parameterization, i.\,e., we simply write $\gamma$ instead of $\gamma(I)$ when no confusion is possible
		(for instance when we integrate over $\gamma(I)$).
		\item For any Lipschitz continuous function $f:(X,d_X)\to (Y,d_Y)$ from one metric space to another we denote the \emph{Lipschitz constant} by $\text{Lip}(f)=\sup_{x_1\neq x_2}d_Y(f(x_1),f(x_2))/d_X(x_1,x_2)$.
				%\item The \emph{Euclidean distance} between two sets $A,B\subset\R^n$ is denoted
				%\begin{equation*}
				%\text{dist}(A,B)=\inf_{x\in A,y\in B}|x-y|.
				%\end{equation*}
				%We use the notation $\text{dist}(x,B)=\text{dist}(B,x)=\text{dist}(\{ x \},B)$.
		%		We write $\text{diam}(A)$ for the diameter of $A$, i.\,e.,
		%		\begin{equation*}
		%		\text{diam}(A)=\sup_{x,y\in A}|x-y|.
		%		\end{equation*}
		%		Moreover, $d_H(A,B)$ denotes the \emph{Hausdorff-distance} between $A$ and $B$, given by
		%		\begin{equation*}
		%		d_H(A,B)=\max\left(\sup_{x\in A}\text{dist}(x,B),\sup_{y\in B}\text{dist}(A,y)\right),
		%		\end{equation*}
		%		where we use the notation $\text{dist}(x,B)=\text{dist}(B,x)=\text{dist}(\{ x \},B)$.
		\item For any set $A$ we write $\iota_A$ for the \emph{indicator function} and $1_A$ for the \emph{characteristic function},
		\begin{equation*}
		\iota_A(x)=\begin{cases*}
		0&if $x\in A$,\\
		\infty&else,
		\end{cases*}
		\qquad
		1_A(x)=\begin{cases*}
		1&if $x\in A$,\\
		0&else.
		\end{cases*}
		\end{equation*}
		\item For $x,y\in\mathcal{C}$ we define $[x,y]$ as the \emph{line segment} $\{ x+t(y-x)\,|\, t\in [0,1] \}$. The sets $(x,y],[x,y)$ and $(x,y)$ are defined similarly, e.\,g., $(x,y]=[x,y]\setminus\{ x \}$.
		\item For any function $f:X\to V$ with values in some normed vector space $(V,\|.\|)$ and $A\subset X$ we write
		\begin{equation*}
		|f|_{\infty,A}=\sup_{x\in A}\|f(x)\|.
		\end{equation*}
		%More precisely, $|.|_{\infty,A}$ corresponds to the \emph{uniform norm} (if restricted to bounded functions).
		\item The \emph{effective domain} of a convex function $f:\R\to\R\cup\{ \infty\}$ is denoted $\text{dom}(f)=\{ x\in\R\,|\, f(x)<\infty \}$.
		\item We write the \emph{convex conjugate} of a function $f:\R\to\R\cup\{ \infty\}$ as
		\begin{equation*}
		f^*(x)=\sup_{m\in\R}mx-f(m).
		\end{equation*}
		\item A sequence $x:\mathbb{N}\to M$ of elements in some set $M$ will be indicated by the notation $(x_i)\subset M$ with $x_i=x(i)$. %If $x_i$ actually stems from a subset $M_i$ we instead speak of a sequence $x_i\in M_i$.
		\item If a sequence of Lipschitz paths $\gamma_j:I\to\R^n$ converges uniformly to some $\gamma$, then we write $\gamma_j\rightrightarrows\gamma$. 
		\item For two functions $f,g:M\to\R$ on some set $M$ we write $f\wedge g=\min\{ f,g\}$ and $f\vee g=\max\{ f,g\}$ for the pointwise minimum/maximum.
		%\item To make the estimates more readable we use the norm $|(x,y)|=|x|+|y|$ on $\R^n\times\R^n$, where at the same time $|.|$ denotes the Euclidean norm on $\R^n$. It will be clear from the context which norm is meant.
	\end{itemize}
	\section{Beckmann formulation of $\boldsymbol{W_{d_{S,a,b}}}$ using duality}
	\label{subs6}
	%\todo[inline,color=yellow]{Note that we use the symbol $L$ for a space and a functional at the same time. Nevertheless, I think this is not confusing. I don't like the idea to change the notation of the functional since we also use it in the first article. Moreover, an alternative notation for the space $L$ would not be likely: we already use $\mathcal{L}$ for a measure and $\mathscr{L}$ in the next main section. The most acceptable idea would be to replace the spaces $L$ and $D$ by $\mathscr{L}$ and $\mathscr{D}$.}
	Let $\mu_+,\mu_-:\mathcal{B}(\mathcal{C})\to[0,1]$ be probability measures. Further, assume that $(S,a,b)$ is a city. Throughout this section, we abbreviate the generalized urban metric $d=d_{S,a,b}$. Recall that $d(x,y)$ is the cost to travel from $x$ to $y$ in terms of the (friction) coefficients $a$ and $b:S\to[0,\infty)$. We extend $b$ to $\mathcal{C}\setminus S$ with value $a$ so that we may write
	\begin{equation*}
	d(x,y)=\inf_{\gamma\in\Gamma^{xy}}\int_{\gamma}b\,\mathrm{d}\mathcal{H}^1.
	\end{equation*}
	In \cref{subs61} we prove \cref{KRformula}, i.\,e., under \cref{assumpmetric} we have
	\begin{equation*}
	W_d(\mu_+,\mu_-)=\sup_{\varphi\in C_d^1}\int_{\mathcal{C} }\varphi\,\mathrm{d}(\mu_+-\mu_-),
	\end{equation*}
	where the set $C_d^1$ is defined by
	\begin{equation*}
	C_d^1=\{\varphi\in C^1(\mathcal{C} )\,|\,|\nabla\varphi |\leq b\text{ in }\mathcal{C}\}.
	\end{equation*}
	To this end, we show that $C_d^1$ is dense (with respect to $|.|_{\infty,\mathcal{C}}$) in 
	\begin{equation*}
	C_d=\{\varphi\in C(\mathcal{C})\,|\,|\varphi(x)-\varphi (y)|\leq d(x,y)\text{ for all }x,y\in\mathcal{C} \}
	\end{equation*}
	and then essentially apply the classical Kantorovich\textendash Rubinstein formula \cite[Thm.\,4.1]{Edw}. Note that the Lipschitz condition $|\varphi(x)-\varphi (y)|\leq d(x,y)$ does not directly imply continuity if $a=\infty$. Further, observe that we have $C_d^1\subset C_d$. Indeed, for $\varphi\in C_d^1$ we get
		\begin{equation*}
		|\varphi (x)-\varphi (y)|\leq\int_\gamma |\nabla\varphi |\,\mathrm{d}\Ha\leq\int_{\gamma}b\,\mathrm{d}\Ha
		\end{equation*}
		for all $x,y\in\mathcal{C}$ and $\gamma\in\Gamma^{xy}$. By taking the infimum over all such $\gamma$ we conclude
		\begin{equation*}
		|\varphi (x)-\varphi (y)|\leq d(x,y),
		\end{equation*}
		which implies $\varphi\in C_d$. 
	In \cref{subs62} we will then prove \cref{BWfinal}, the Beckmann formulation of the Wasserstein distance $W_d(\mu_+,\mu_-)$, under \cref{assumpmetric} or \cref{afinite}.
	We will occasionally write
	\begin{equation*}
	L(\gamma)=\int_Ib(\gamma)|\dot{\gamma}|\mathrm{d}\mathcal{L}
	\end{equation*}
	for the length of a Lipschitz path $\gamma:I\to\mathcal{C}$ associated with $d$.
	
	%For $B\in\mathcal{B}(\R^n)$ and any nonnegative Borel measurable function $f:B\to\R$ we write
	%\begin{equation*}
	%d_f(x,y)=\inf_{\gamma}\int_{\gamma\cap B}f\,\mathrm{d}\Ha+a\Ha(\gamma\setminus B),
	%\end{equation*}
	%where the infimum is taken over Lipschitz paths $\gamma:I\to\R^n$ with $\gamma(0)=x$ and $\gamma(1)=y$. Moreover, we define
	%\begin{equation*}
	%L_f(\eta)=\int_{\eta^{-1}(B)}f(\eta)|\dot{\eta}|\,\mathrm{d}\mathcal{L}+a\int_{\eta^{-1}(\R^n\setminus B)}|\dot{\eta}|\,\mathrm{d}\mathcal{L}
	%\end{equation*}
	%for all Lipschitz paths $\eta:I\to\R^n$. % Note that we consider paths and domains of definition which do not necessarily lie entirely in $\mathcal{C}$. This aspect has technical reasons: we will extend certain functions on $S$ to $\delta$-neighbourhoods of $S$.
	%\todo[inline]{It is somewhat unsatisfying that the $b$ was replaced by $f$, but the $a$ stayed the same, in particular since $b$ was previously extended by $a$; it is also unsatisfying that $B$ is implicitly defined as the domain of $f$ (then e.g. in case of $f=b$ you ask yourself whether the extended or the nonextended $b$ is meant) -- can one not get rid of $B$ and the $a$-term in this notation?}
	%\todo[inline,color=yellow]{We don't need the notations. Define $L$ in this introduction and replace the remark below.}
	\begin{rem}[Formula for $d$]
		It is easy to see that (cf.\ \cite[Lemma\,2.2.1]{LSW})
		\begin{equation*}
		d(x,y)=\inf_{\gamma\in\Gamma^{xy}}L(\gamma).
		\end{equation*}
	\end{rem}
	\subsection{Kantorovich\textendash Rubinstein duality in urban planning}
	\label{subs61}
	%When studying Kantorovich\textendash Rubinstein duality, the following assumption (weak sufficient condition for $d$ to be a metric) is reasonable.
	%\begin{assump}[Positive bounds for friction coefficients]
		%\label{dassump}
		%Throughout \cref{subs61} we assume that $b$ is bounded away from $0$ and $a\in(0,\infty)$.
	%\end{assump}
	In this section we prove \cref{KRformula}. To this end, we elaborate a version of the Stone\textendash Weierstraß theorem to work out that $C_d^1$ is a dense subset of $C_d$. The first step is to provide a point separation statement for $C_d^1$ (\cref{psepa}). We will need to approximate the urban metric $d$. For this purpose, we define $b_k:\mathcal{C}\to[0,\infty)$ by
	\begin{equation*}
	b_k(z)=k\wedge\min_{\mathcal{C}\cap\overline{B}_{1/k}(z)}b
	\end{equation*}
	and abbreviate $d_k=d_{S,a,b_k}$ (which is well-defined if $b_k$ is lower semi-continuous, see \cref{papprox} below). Moreover, we write $L_k$ for the path length associated with $d_k$.
	\begin{prop}[Pointwise approximation of urban metric]
	\label{papprox}
	Let \cref{assumpmetric} be satisfied. Then the $b_k$ are lower semi-continuous, and we have $d_k\to d$ pointwise.
	\end{prop}
	\begin{proof}
	Fix $k$ and let $(z_j)\subset\mathcal{C}$ be a sequence such that $z_j\to z\in\mathcal{C}$. We can assume that $\liminf_jb_k(z_j)=\lim_jb_k(z_j)<\infty$ so that every subsequence of $j\mapsto b_k(z_j)$ has the same limit. For every $z_j$ there exists some $\tilde{z}_j\in \mathcal{C}\cap \overline{B}_{1/k}(z_j)$ such that $b_k(z_j)=k\wedge b(\tilde{z}_j)$. By the compactness of $ \mathcal{C}$ we have $\tilde{z}_j\to\tilde{z}\in\mathcal{C}$ up to a subsequence. Note that $|\tilde{z}_j-z|\leq 1/k+|z_j-z|$ and thus $\tilde{z}\in \mathcal{C}\cap \overline{B}_{1/k}(z)$ by letting $j\to\infty$. This yields
	\begin{equation*}
	b_k(z)\leq k\wedge b(\tilde{z})\leq\liminf_j(k\wedge b(\tilde{z}_j))=\liminf_jb_k(z_j),
	\end{equation*}
	where we used the lower semi-continuity of $b$. It remains to prove that $d_k\to d$ pointwise. Let $(x,y)\in\mathcal{C}\times\mathcal{C}$ be arbitrary with $x\neq y$. Since $d_k\leq d$ it sufficies to show $d(x,y)\leq\liminf_kd_k(x,y)$ under the assumption that the right hand side is finite. Let $(\gamma_k)\subset\Gamma^{xy}$ be a sequence of Lipschitz paths such that $L_k(\gamma_k)\leq d_k(x,y)+1/k$. Using that $b$ is bounded away from zero we obtain
	\begin{equation*}
	0<\textup{len}(\gamma_k)\inf_\mathcal{C}b\leq L_k(\gamma_k)\leq const.<\infty
	\end{equation*}
	for all $k>\inf b$. Thus, the lengths of the $\gamma_k$ are uniformly bounded. We can thus reparameterize each $\gamma_k$ by arc length and get $\gamma_k\rightrightarrows\gamma:I\to\mathcal{C}$ up to a subsequence by the Arzel\`a\textendash Ascoli theorem. This also yields $|\dot{\gamma}|\leq\liminf_k|\dot{\gamma}_k|$ almost everywhere in $I$. Now fix any $t\in I$. Once more we get $b_k(\gamma_k(t))=k\wedge b(z_k)$ for some $z_k\in\mathcal{C}\cap\overline{B}_{1/k}(\gamma_k(t))$. Clearly, we must have $z_k\to\gamma(t)$. Using again the lower semi-continuity of $b$ we estimate
	\begin{equation*}
	b(\gamma(t))\leq\liminf_k(k\wedge b(z_k))=\liminf_kb_k(\gamma_k(t)).
	\end{equation*}
	Finally, using Fatou's lemma
	\begin{equation*}
	d(x,y)\leq\int_Ib(\gamma)|\dot{\gamma}|\mathrm{d}\mathcal{L}\leq\int_I\liminf_kb_k(\gamma_k)\liminf_k|\dot{\gamma}_k|\mathrm{d}\mathcal{L}\leq \liminf_k\int_Ib_k(\gamma_k)|\dot{\gamma}_k|\mathrm{d}\mathcal{L}=\liminf_kL_k(\gamma_k)\leq\liminf_kd_k(x,y)+\frac{1}{k},
	\end{equation*}
	which shows the pointwise convergence $d_k\to d$.
	\end{proof}
	The following two examples illustrate that the conditions on $b$ in \cref{papprox} cannot be dropped.
	\begin{examp}[Counterexample for $b$ not lower semi-continuous in $\mathcal{C}$]\label{exm:pointwiseDistanceApproximationNotLSC}
	\label{bnotlsc}
	Take $S=\mathbb{Q}\cap [0,1]$ and $a,b\in (0,\infty ]$ with $b<a$, which implies that the extension of $b$ with value $a$ to $\mathcal{C}$ is not lower semi-continuous (see \cref{sf:bnotlowsc}). Then, we obtain $d(0,1)=a>b=d_k(0,1)$ since $b_k=b$ in $[0,1]$ for all $k$.
	\end{examp}
	\begin{figure}[t]
		\centering
		\begin{subfigure}[b]{0.5\textwidth}
			\centering
			\begin{tikzpicture}[scale=3.6]
			\foreach \x in {0,1,...,50}
			\node[circle,fill=black,inner sep=0.8pt,minimum size=0.01cm] at ({\x/50},0) {};
			\foreach \y in {0,1,...,50}
			\node[circle,fill=black,inner sep=0.8pt,minimum size=0.01cm] at ({\y/50},0.3) {};
			\draw [color=black, line width = 0.07cm, domain=-0:1] plot (\x,{1});
			\node[left] at (0,1) {$a$};
			\draw[-stealth] (-0.1,0) -- (1.1,0);
			\draw[-stealth] (0,-0.1) -- (0,1.1) node[left] {$b$};
			\draw (1,-0.013) -- (1,0.013) node[below] {$1$};
			\node[label={-90: $S=\mathbb{Q}\cap[0,1]$}] at (0.5,0) {};
			\end{tikzpicture}
			\caption{Extended $b$ not lower semi-continuous.}
			\label{sf:bnotlowsc}
		\end{subfigure}\begin{subfigure}[b]{0.5\textwidth}
			\centering
			\begin{tikzpicture}[scale=1.4]
			\draw[line width=0.04cm] (0,1) -- (1,2) node[above,xshift=-1.4cm,yshift=-0.6cm] {$(S_1,b=1)$};
			\draw[gray] (0,-1) -- (0,1);
			\draw[gray] (1,{sin(1 r)}) -- (1,2);
			\draw[gray,smooth,samples=1000, domain=0.01:1] plot(\x, {sin(1/\x r)}) node[right] {$(S_2,b=0)$};
			\node[circle,fill=gray,inner sep=0.5pt,minimum size=0.1cm,label={180: $x$}] at (0,1) {};
			\node[circle,fill=gray,inner sep=0.5pt,minimum size=0.1cm,label={0: $y$}] at (1,2) {};
			\end{tikzpicture}	
			\caption{$b$ not bounded away from zero.}
			\label{sf:bGetsZero}
		\end{subfigure}
		\caption{Sketches for \cref{bnotlsc,bbounded}.}
	\end{figure}
	\begin{examp}[Counterexample for $b$ not bounded away from zero]\label{exm:pointwiseDistanceApproximationNotPositive}
	\label{bbounded}
	Let $x=(0,1),y=(1,2)$, and $S=S_1\cup S_2$ with $S_1=(x,y)$ and $ S_2=[-x,x]\cup\{ (t,\sin(1/t))\,|\, t\in(0,1]\}\cup[(1,\sin(1)),y]$, see \cref{sf:bGetsZero}. Set $a=\infty$, $b=1$ on $S_1$, and $b=0$ on $S_2$. Note that $b$ (continued with value $a$ outside $S$) is lower semi-continuous. Moreover, we have $d(x,y)=\sqrt{2}$, because it is impossible to reach $y$ from $x$ using $S_2$ without travelling on the complement of $S$, which produces an infinite cost. However, we have $d_k(x,y)=0$: Choose $t_k\in(0,1]$ such that $z_k=(t_k,\sin(1/t_k))\in B_{1/k}(x)$. Then any injective Lipschitz path $\gamma$ on $[x,z_k]\cup\{ (t,\sin(t))\,|\, t\in(t_k,1] \}\cup [(1,\sin(1)),y]$ from $x$ to $y$ satisfies $L_k(\gamma)=0$.
	\end{examp}
	An immediate consequence is the lower semi-continuity of the urban metric.
	\begin{corr}[Lower semi-continuity of urban metric]\label{thm:lscMetric}
	Under \cref{assumpmetric} $d$ is lower semi-continuous.
	\end{corr}
	\begin{proof}
	The (pseudo-)metric $d$ is the pointwise supremum of the Lipschitz-continuous metrics $d_k$.
	\end{proof}
	\begin{examp}[Counterexample if $b$ is not bounded away from zero]
	For $a<\infty$ it is not difficult to see that the urban metric $d$ is continuous even without \cref{assumpmetric} \cite[Prop.\,2.2.3]{LSW}. In the situation of \cref{exm:pointwiseDistanceApproximationNotPositive}, where $b$ was not bounded away from zero and $a=\infty$, we get that $d$ is not lower semi-continuous. Indeed, letting $(x_j)\subset S_2\backslash [-x,x]$ with $x_j\to x$ we obtain
	\begin{equation*}
	d(x,y)=\sqrt{2}>0=\liminf_jd(x_j,y).
	\end{equation*}	 
	\end{examp}
	%\todo[inline,color=yellow]{This works for $a=\infty$ and $b_k(z)=k\wedge\min_{\overline{B}_{1/k}(z)}b$ as well as for $b_k(z)=\min_{\overline{B}_{1/k}(z)}b$. In the second case, the argument is not the same and $d_k$ is in general not Lipschitz.}
	\begin{prop}[Point separation with tolerance]
	\label{psepa}
	Let \cref{assumpmetric} be satisfied. Then for arbitrary $x,y\in\mathcal{C}$ and $t_1,t_2\in\R$ with $|t_2-t_1|<d(x,y)$ there is a function $f\in C_d^1$ such that $f(x)\leq t_1$ and $f(y)\geq t_2$.
	\end{prop}
	%\todo[inline,color=yellow]{$g_k$ is not Lipschitz for $a=\infty$ and $b_k(z)=\min_{\overline{B}_{1/k}(z)}b$. For $a=\infty$ and $b_k(z)=k\wedge\min_{\overline{B}_{1/k}(z)}b$ we get $|\nabla f|\leq k\wedge b$ for some $k$. Hence, we have $f\in D_k=\{ g\in C^1(\mathcal{C})\,|\,|\nabla g|\leq k\wedge b\}\subset D$, which improves the previous result.}
	\begin{proof}
	If $t_1>t_2$, then one can take $f\equiv(t_1+t_2)/2$, so we may assume $t_2\geq t_1$.
	Further, it is enough to prove the claim for $(t_1,t_2)$ replaced by $(0,t)$ with $0<t<d(x,y)$. Indeed,
	given some $\tilde f\in C_d^1$ with $\tilde f(x)\leq0$ and $\tilde f(y)\geq t$ for the choice $t=t_2-t_1$,
	the desired $f$ is obtained as $t_1+\tilde f$.
	
	We will now mollify the functions $g_k(z)=d_k(x,z)$. For this purpose, we assume that $g_k$ is defined on $\R^n$, or equivalently, for the rest of the proof we assume that 
	\begin{center}
	$b$ is extended with value $a$ to $\R^n$,\qquad $b_k(z)=k\wedge\min_{\overline{B}_{1/k}(z)}b$,\qquad and\qquad $d_k(z_1,z_2)=\inf_\gamma\int_Ib_k(\gamma)\mathrm{d}\mathcal{L}$,
	\end{center}		
	 where the infimum is over Lipschitz paths $\gamma:I\to\R^n$ with $\gamma(0)=z_1$ and $\gamma(1)=z_2$.	
	 Clearly, the functions $g_k$ are Lipschitz with $\textup{Lip}(g_k)\leq k$. Now let $\eta$ be a smooth mollifier on $\R^n$ with unit integral and support on $B_1(0)$. We define $g_k^\varepsilon=\eta_\varepsilon*g_k$ for the scalings $\eta_\varepsilon(z)=\varepsilon^{-n}\eta(\varepsilon^{-1}z)$. Then the functions $f_k^\varepsilon=g_k^\varepsilon-g_k^\varepsilon(x)$ satisfiy $f_k^\varepsilon(x)=0$ and $f_k^\varepsilon (y)\to g_k(y)$ as $\varepsilon\to0$ due to the pointwise convergence $g_k^\varepsilon\to g_k$ and $g_k(x)=0$. Now fix $k$ sufficiently large so that $t<g_k(y)=d_k(x,y)<d(x,y)$ (possible by \cref{papprox}) and then $\varepsilon$ small enough such that $t\leq f_k^\varepsilon (y)$. It remains to prove that the restriction of $f=f_k^\varepsilon$ to $\mathcal{C}$ is an element of $C_d^1$. By definition $f$ is smooth and we have $\nabla f=\nabla g_k^\varepsilon$. Let $z\in\mathcal{C}$ be arbitrary and choose $\varepsilon<1/k$ (if not a priori satisfied). We then have $b_k\leq k\wedge b(z)\leq b(z)$ in $B_\varepsilon(z)$. Now let $y\in B_\varepsilon(z)$ be any point of total differentiability of $g_k$. Then for all $\nu\in\mathcal{S}^{n-1}$ and $h>0$ sufficiently small such that $[y,y+h\nu]\subset B_\varepsilon(z)$ we have
\begin{equation*}
g_k(y+h\nu)-g_k(y)\leq d_k(y,y+h\nu)\leq\int_{[y,y+h\nu]}b_k\,\mathrm{d}\Ha\leq hb(z)
\end{equation*} 
by the triangle inequality and the choice of $\varepsilon,h$. It follows that the directional derivative $\partial_\nu g_k(y)$ is bounded by $b(z)$. By the arbitrariness of $y$ and $\nu$ we thus obtain $|\nabla g_k(y)|\leq b(z)$ for $\Le^n$-almost every $y\in B_\varepsilon(z)$. Using this we obtain the final estimate
\begin{equation*}
|\nabla f(z)|=|\nabla g_k^\varepsilon(z)|\leq\int_{B_\varepsilon(z)}\eta_\varepsilon(z-y)|\nabla g_k(y)|\mathrm{d}\Le^n(y)\leq b(z)\int_{B_\varepsilon(z)}\eta_\varepsilon(z-y)\mathrm{d}\Le^n(y)=b(z),
\end{equation*}	
which shows that $f$ is an element of $C_d^1$.
	\end{proof}
	\begin{rem}[Regularity of $f$]
	Following the proof of \cref{psepa} we actually have $f\in C_{d_k}^1\subset C_d^1$. It can be shown that \cref{papprox} and \cref{thm:lscMetric} are also true for $b_k$ replaced by $z\mapsto\min_{\mathcal{C}\cap\overline{B}_{1/k}(z)}b$. However, using these functions, the mollfications of $z\mapsto d_k(x,z)$, which we employ in the following, would not be real-valued.
	\end{rem}
	\notinclude{
	\begin{rem}[Approximation of urban metric]
	Another strategy to approximate the urban metric (cf.\ \cref{papprox}) to provide an appropriate point separation statement as in \cref{psepa} is to approximate the friction coefficient $b|_S$ from below using Lipschitz continuous functions $b_k:S\to[0,\infty)$ (for example defined via the Moreau envelope). Then, if $b$ is integrable with respect to $\Ha\mres S$, it is straightforward to check that $d_{b_k}\rightrightarrows d$ on the set of points where $d$ is finite. Further, one may continue the $b_k$ to Lipschitz continuous functions $b_k^\delta$ on the tubular neighbourhood $S_\delta=\{ z\in\R^n\,|\,\textup{dist}(z,S)\leq\delta \}$ for $\delta>0$ and prove that $d_{b_k^\delta}\rightrightarrows d_{b_k}$ for $\delta\to 0$ (which can for example be realized if $0<\inf b\leq a<\infty$ and $S$ closed). The mollification approach in \cref{psepa} can then be applied to the function $z\mapsto d_{b_k^\delta}(x,z)$.
	\end{rem}
	}%\notinclude
	We can now show a version of the Stone\textendash Weierstraß theorem which states that $C_d^1$ is a dense subset of $C_d$. We follow the proofs in \cite[Appendix A]{SW}.
	For $f,g\in C^1(\mathcal{C})$ we will need to approximate $f\wedge g$ and $f\vee g$ by smooth functions,
	for which we require the following property of mollified Heaviside step functions.
	\begin{lem}[{{\cite[Lem.\,A.3]{SW}}}]
		\label{heaviside}
		For all $\delta>0$ there is a monotone (smoothed Heaviside step) function $H_\delta\in C^\infty (\R)$ such that $H_\delta=0$ on $(-\infty ,-\delta]$, $H_\delta=1$ on $[\delta,\infty )$ and $|tH_\delta'(t)|\leq\delta$ for all $t\in\R$.
	\end{lem}
	Let $H_\delta$ be as in \cref{heaviside}. We define 
	\begin{equation*}
	f\wedge_\delta g=H_\delta (f-g)g+(1-H_\delta(f-g))f\qquad\text{and}\qquad
	f\vee_\delta g=H_\delta (f-g)f+(1-H_\delta (f-g))g
	\end{equation*}
	for all $\delta >0$ and $f,g\in C^1(\mathcal{C})$. Then $f\wedge_\delta g,f\vee_\delta g\in C^1(\mathcal{C})$ are approximations of $f\wedge g$ and $f\vee g$ with respect to $|.|_{\infty,\mathcal{C}}$.
	%\todo[inline,color=yellow]{This statement works (of course) also with $D$ replaced by $D_k$.}
	\begin{lem}[{{Smooth min and max operation, cf.\ \cite[Lem.\,A.4]{SW}}}]
		\label{minmaxop}
		Let $\delta>0$ and $f,g\in C_d^1$. Then we have $f\wedge_\delta g,f\vee_\delta g\in (1+2\delta) C_d^1$ as well as $|f\wedge_\delta g-f\wedge g|,|f\vee_\delta g-f\vee g|\leq\delta$ in $\mathcal{C}$.
	\end{lem}
	\begin{proof}
		We have $f\wedge_\delta g\in C^1(\mathcal{C} )$ by definition. Furthermore, we can estimate
		\begin{align*}
		|\nabla (f\wedge_\delta g)|&=|H_\delta (f-g)\nabla g+(1-H_\delta (f-g))\nabla f+(g-f)H_\delta'(f-g)\nabla (f-g)|\\&\leq H_\delta (f-g)|\nabla g|+(1-H_\delta (f-g))|\nabla f|+|(g-f)H_\delta'(f-g)|(|\nabla f|+|\nabla g|)\\
		&\leq(1+2\delta) b,
		\end{align*}
		thus $f\wedge_\delta g\in (1+2\delta ) C_d^1$. Additionally, we have $f\wedge_\delta g=f\wedge g$ if $|f-g|\geq\delta$. The function $f\wedge_\delta g$ is a convex combination of $f$ and $g$ , therefore, we conclude $|f\wedge_\delta g-f\wedge g|\leq\delta$. The proof for $f\vee_\delta g$ is analogous.
	\end{proof}
	We can now use the operations $\wedge_\delta$ and $\vee_\delta$ to obtain ``$C^1$-gluings'' of functions on certain open covers of $\mathcal{C}$.
	\begin{prop}[{{Version of the Stone\textendash Weierstraß theorem, cf.\ \cite[Prop.\,A.5]{SW}}}]
		\label{stone}
		Let \cref{assumpmetric} be satisfied. Then $C_d^1$ is a dense subset of $C_d$ with respect to the uniform norm $|.|_{\infty ,\mathcal{C} }$.
	\end{prop}
	\begin{proof}
	 We have already seen that $C_d^1$ is a subset of $C_d$ in the introduction of \cref{subs6}. For the denseness, let $\varepsilon>0$ and $g\in C_d$. By $\lambda\tilde{g}\rightrightarrows\tilde{g}$ as $\lambda\nearrow 1$ for all $\tilde{g}\in C_d$ we can assume $g\in\lambda C_d$ for some $\lambda\in (0,1)$. Now let $x\in\mathcal{C}$ be arbitrary. For each $y\in\mathcal{C}\setminus\{ x\}$ we have $d(x,y)=\infty$ or $|g(x)-g(y)|\leq\lambda d(x,y)<d(x,y)$ so that by \cref{psepa} (point separation) there exists a function $f_y\in C_d^1$ with $f_y(x)\geq g(x)$ and $f_y(y)\leq g(y)$. Define the (relatively) open sets $V_y=\{ f_y<g+\varepsilon/4\}$ ($f_y$ and $g$ are continuous). There exists a finite cover $V_{y_1},\ldots ,V_{y_k}$ of $\mathcal{C}$ by the compactness of $\mathcal{C}$ (Heine\textendash Borel theorem). For $\delta>0$ define $\tilde{F}_x=(\ldots ((f_{y_1}\wedge_\delta f_{y_2})\wedge_\delta f_{y_3})\ldots\wedge_\delta f_{y_k})$. Then $\tilde{F}_x\in (1+2\delta)^k C_d^1$ and $|\tilde{F}_x-\min\{ f_{y_1},\ldots ,f_{y_k}\} |\leq k\delta$ by \cref{minmaxop} and induction.
		Now assume without loss of generality that $g,f_{y_1},\ldots,f_{y_k},\tilde F_x$ are all nonnegative (otherwise we can just add the same, sufficiently large constant to all of them).
		The function $F_x=(1+2\delta )^{-k}\tilde{F}_x\in C_d^1$ satisfies
		\begin{equation*}
		F_x\leq (1+2\delta)^{-k}(\min\{ f_{y_1},\ldots ,f_{y_k}\} +k\delta )<(1+2\delta )^{-k}(g+\varepsilon/4+k\delta )<g+\varepsilon/4+k\delta<g+\varepsilon/2
		\end{equation*}
		for $\delta$ sufficiently small. Now let $W_x=\{ F_x>g-\varepsilon/4\}$. Then we obtain
		\begin{equation*}
		F_x(x)\geq (1+2\delta )^{-k}(\min\{ f_{y_1}(x)\ldots ,f_{y_k}(x)\} -k\delta )\geq (1+2\delta )^{-k}(g(x)-k\delta )%\\&\geq (1+2\delta )^{-k}(g(x)-\varepsilon/8)-(1+2\delta)^{-k}k\delta
		> g(x)-\varepsilon/4
		\end{equation*}
		for $\delta$ sufficiently small. Thus, we have $x\in W_x$. Again, there exists a finite (relatively) open cover $W_{x_1},\ldots ,W_{x_l}$ of $\mathcal{C}$. For $\tilde{\delta}>0$ and $\tilde{f}=(\ldots ((F_{x_1}\vee_{\tilde{\delta }} F_{x_2})\vee_{\tilde{\delta }} F_{x_3})\ldots\vee_{\tilde{\delta }} F_{x_l})$ we have $f=(1+2\tilde{\delta})^{-l}\tilde{f}\in C_d^1$,
		\begin{equation*}
		f\leq (1+2\tilde{\delta})^{-l}(\max\{ F_{x_1},\ldots ,F_{x_l}\}+l\tilde{\delta})\leq (1+2\tilde{\delta})^{-l}(g+\varepsilon/2+l\tilde{\delta})\leq g+\varepsilon/2+l\tilde{\delta}\leq g+\varepsilon,
		\end{equation*}
		and
		\begin{equation*}
		f\geq (1+2\tilde{\delta})^{-l}(\max\{F_{x_1},\ldots ,F_{x_l}\} -l\tilde{\delta})\geq (1+2\tilde{\delta})^{-l}(g-\varepsilon/4-l\tilde{\delta})%\\&\geq (1+2\tilde{\delta})^{-l}(g-\varepsilon/8)-(1+2\tilde{\delta})^{-l}\varepsilon/4-(1+2\tilde{\delta})^{-l}l\tilde{\delta}
		\geq g-\varepsilon
		\end{equation*}
		for $\tilde{\delta }$ sufficiently small using also that $g$ is bounded. In conclusion, we get $f\in C_d^1$ and $|f-g|_{\infty ,\mathcal{C} }\leq\varepsilon$.
	\end{proof}
	\begin{examp}[Counterexample for $b$ not lower semi-continuous]
		\label{exScl}
		\Cref{stone} is in general not true without the condition that $b$ is lower semi-continuous in $\mathcal{C}$. Assume that $(S,a,b)$ are as in \cref{exm:pointwiseDistanceApproximationNotLSC} and $a<\infty$. Then we have $d(x,y)=a|x-y|$. The function $g(z)=az$ satisfies $|g(x)-g(y)|=d(x,y)$ for all $x,y\in [0,1]$, thus $g\in C_d$. However, for $\varepsilon>0$ sufficiently small, no $f\in C_d^1$ can satisfy the inequality $|f-g|_{\infty ,\mathcal{C} }<\varepsilon$. To see this, choose some $\varepsilon<(a-b)/2$. Assume that there exists some $f\in C_d^1$ with $|f-g|_{\infty ,\mathcal{C} }<\varepsilon$, then $|f'|\leq b$ on $[0,1]$ by continuity and thus $f(1)\leq f(0)+b\leq\varepsilon+b<a-\varepsilon=g(1)-\varepsilon$, which is a contradiction. 
	\end{examp}
	\begin{rem}
	We believe that it is possible to remove the assumption that $b$ is bounded away from zero in \cref{stone}. A strategy to prove this might be as follows. Let $D(x,y)=\lim_kd_k(x,y)$ (which would be equal to $d(x,y)$ under \cref{assumpmetric}). By \cref{psepa} and \cref{stone} it would then be sufficient to prove that $C_d=\{ \varphi\in C(\mathcal{C})\,|\, |\varphi(x)-\varphi(y)|\leq D(x,y)\}$ which would follow if we had $D(x,y)=\sup_{\phi\in C_d}|\phi(x)-\phi(y)|$.  
	\end{rem}
	Our main result of this subsection now follows from the previous density result and the classical Kantorovich\textendash Rubinstein duality argument.% formula \cite[Thm.\,4.1]{Edw}.
	% in fact, if d is lower semi-continuous, then by Santambrogio Thm.1.42 we have strong duality, and by Villani (Old and New) Particular Case 5.4 we can replace each dual variable by its d-transform (which will still be continuous), thereby increasing the energy, and the Rubinstein formula follows
	%\todo[inline,color=yellow]{Idea with classical KR-formula (not successful): Using our approximations $d_k$ we clearly have $W_d(\mu_+,\mu_-)\geq\sup_kW_{d_k}(\mu_+,\mu_-)$. For the reverse inequality assume that $\sup_kW_{d_k}(\mu_+,\mu_-)$ is finite and let $\pi_k$ be some corresponding optimal transport plans. By Prokhorov's theorem we have $\pi_k\xrightharpoonup{*}\pi$ up to a subsequence. But we do not have that the $d_k$ converge uniformly to $d$. If we can show that $W_d(\mu_+,\mu_-)\leq\sup_kW_{d_k}(\mu_+,\mu_-)$, then the desired result would follow from the classical KR-formula: $$W_d(\mu_+,\mu_-)\leq\sup_kW_{d_k}(\mu_+,\mu_-)=\sup_k\sup_{\varphi\in D_k}\int\varphi\mathrm{d}(\mu_+-\mu_-)=\sup_D\int\varphi\mathrm{d}(\mu_+-\mu_-),$$where we used $D=\bigcup_k D_k$ (indeed, for $f\in D$ choose $k\geq |\nabla f|_{\infty,\mathcal{C}}$).}
	\begin{proof}[Proof of \cref{KRformula}]
		Since $d$ is bounded below and lower semi-continuous by \cref{thm:lscMetric},
		the dual formulation of the Wasserstein distance is known to be (see for instance \cite[Thm.\,1.42]{San} or \cite[Thm.\,5.10(i)]{Villani})
		\begin{equation*}
		W_d(\mu_+,\mu_-)=\sup\left\{\int_{\mathcal C}\phi\,\d\mu_+-\int_{\mathcal C}\psi\,\d\mu_-\,\middle|\,\phi,\psi\in C(\mathcal C),\,\phi(x)-\psi(y)\leq d(x,y)\,\forall x,y\in\mathcal C\right\}.
		\end{equation*}
		Moreover, it is a classical argument (compare with the condition in the above supremum) that given $\phi$ the optimal $\psi$ is equal to (functions of this form are called $d$-convex)
		\begin{equation*}
		\psi(y)=\sup_{x\in\mathcal C}\phi(x)-d(x,y),
		\end{equation*}
		so that we may restrict the $\psi$'s in the supremum to $d$-convex functions \cite[Thm.\,5.10(i)]{Villani}. Again it is not difficult to show that a function $\zeta$ is $d$-convex if and only if $|\zeta(x)-\zeta(y)|\leq d(x,y)$ for all $x,y\in\mathcal C$ \cite[Case 5.4]{Villani}. Thus, we obtain
		\begin{equation*}
		W_d(\mu_+,\mu_-)=\sup\left\{\int_{\mathcal C}\phi\,\d\mu_+-\int_{\mathcal C}\psi\,\d\mu_-\,\middle|\,\phi,\psi\in C(\mathcal C),\,\phi(x)-\psi(y)\leq d(x,y),\,|\psi(x)-\psi(y)|\leq d(x,y)\,\forall x,y\in\mathcal C\right\}.
		\end{equation*}
		The condition $\phi(x)-\psi(y)\leq d(x,y)$ now implies $\psi\geq\phi$ so that we may directly assume $\psi=\phi$ without changing the supremum. We therefore end up with $W_d(\mu_+,\mu_-)=\sup_{\psi\in C_d}\int_{\mathcal C}\psi\,\d(\mu_+-\mu_-)$.
		\notinclude{
		The generalized urban metric $d$ and the Euclidean distance induce the same topology by 
		\begin{equation*}
		\inf_Sb|x-y|\leq d(x,y)\leq a|x-y|\text{ for all }x,y\in\R^n.
		\end{equation*}			
		Additionally, $(\mathcal{C} ,d)$ is a Radon space (\cite[Def.\,5.1.4]{AGS} and \cite[p.\,343]{Els}), which allows us to apply the Kantorovich\textendash Rubinstein formula \cite[Thm.\,4.1]{Edw},
		\begin{equation*}
		W_d(\mu_+,\mu_-)=\sup_{\varphi\in L}\int_{\mathcal{C} }\varphi\,\mathrm{d}(\mu_+-\mu_-).
		\end{equation*}
		}%\notinclude
		Furthermore, we have
		\begin{equation*}
		\int_{\mathcal{C} }f\,\mathrm{d}(\mu_+-\mu_-)\leq |f|_{\infty , \mathcal{C} }|\mu_+-\mu_-|(\mathcal{C} )
		\end{equation*}
		for all $f\in C(\mathcal{C} )$ so that the claim now follows from \cref{stone}.
	\end{proof}
	\begin{examp}[Counterexample if $b$ is not lower semi-continuous in $\mathcal{C}$]
	\label{sharp1}
	Assume that $(S,a,b)$ are as in \cref{exm:pointwiseDistanceApproximationNotLSC}. Then for $\mu_+=\delta_0$ and $\mu_-=\delta_1$ we get
	\begin{equation*}
	W_d(\mu_+,\mu_-)=d(x,y)=a>b=\sup_{\phi\in C_d^1}\int_\mathcal{C}\phi\,\mathrm{d}(\mu_+-\mu_-).
	\end{equation*}
	\end{examp}
	\begin{examp}[Counterexample if $b$ is not bounded away from zero]
	\label{sharp2}
	We do not have a Kantorovich\textendash Rubinstein formula without the assumption that $b$ is bounded away from $0$. Indeed, take $\mu_+=\delta_x$ and $\mu_-=\delta_y$ in \cref{bbounded}. Then we directly get
	\begin{equation*}
	W_d(\mu_+,\mu_-)=\sqrt{2}>0=\sup_{\phi\in C_d}\int_\mathcal{C}\phi\,\mathrm{d}(\mu_+-\mu_-)
	\end{equation*}
	since every element in $\phi\in C_d$ is constant on $S_2$ and thus $\phi(x)=\phi(y)$. %Note that this is not a counterexample of $C_d^1$ being a dense subset of $C_d$.
	\end{examp}
	\subsection{Wasserstein distance as min-cost flow using Fenchel-duality}
	\label{subs62}
	In this section we prove \cref{BWfinal} by applying Fenchel's duality theorem \cite[Thm.\,4.4.18]{BV}. We will interpret the Beckmann problem as the dual problem to the Kantorovich--Rubinstein formula from \cref{KRformula}. Consequently, the primal variables lie in $C^1(\mathcal{C})$, whereas the dual variables correspond to Radon measures as in the following \namecref{beckradon}. We abbreviate $B(\mu_+,\mu_-)=B_{S,a,b}(\mu_+,\mu_-)$ (see \cref{BWfinal}).
	\begin{lem}[Version of the Beckmann problem]
		\label{beckradon}
		The Beckmann problem from \cref{BWfinal} is equivalent to a problem on Radon measures, 
		\begin{equation*}
		B(\mu_+,\mu_-)=\inf_{\F}\int_Sb\,\mathrm{d}|\F\mres S|+a|\F|(\mathcal{C}\setminus S)+\iota_{\{\mu_+-\mu_-\}}(\textup{div}\,\F),
		\end{equation*}
		where the infimum is taken over $\F\in\mathcal{DM}^n(\mathcal{C})$.
	\end{lem}
	\begin{proof}
		Clearly, the infimum on the right-hand side is no larger than $B(\mu_+,\mu_-)$. For the reverse inequality we assume that there exists $\F\in\mathcal{DM}^n(\mathcal{C})$ such that $\textup{div}\,\F=\mu_+-\mu_-$. By \cite[Thm.\,3.1]{Sil} we have $\F=\vartheta\Ha\mres M+\G$ for some $M\subset\mathcal{C}$ countably $1$-rectifiable and $\Ha$-measurable, $\vartheta\in L^1(\Ha\mres M;\R^n)$, and an $\Ha$-diffuse vector measure $\G\in\mathcal{M}^n(\mathcal{C})$. Let
		\begin{equation*}
		\xi=
		\begin{cases*}
		\vartheta&on $M\cap S$,\\
		0&on $S\setminus M$
		\end{cases*}
		\end{equation*}
		and $\F^\perp=\F\mres (\mathcal{C}\setminus S)$. Then $\xi\in L^1(\Ha\mres S;\R^n)$, $\F^\perp\in\mathcal{M}^n(\mathcal{C})$ with $\F^\perp\mres S=0$, and
		\begin{equation*}
		\xi\Ha\mres S+\F^\perp=\vartheta\Ha\mres (M\cap S)+\F-\F\mres S=\F,
		\end{equation*}
		because $\F\mres S=\vartheta\Ha\mres (M\cap S)+\G\mres S=\vartheta\Ha\mres (M\cap S)$. Therefore, the measure $\xi\Ha\mres S+\F^\perp$ satisfies the divergence constraint. Finally, we estimate
		\begin{equation*}
		\int_{S}b|\xi|\,\mathrm{d}\Ha+a|\F^\perp|(\mathcal{C} )\leq\int_{S}b\,\mathrm{d}|\F\mres S|+a|\F|(\mathcal{C}\setminus S ).
		%\leq\int_{S}b\,\mathrm{d}|\Xi|+a|\F\mres S|(\mathcal{C})+a|\F\mres (\mathcal{C}\setminus S)|(\mathcal{C} )
		%\int_{S}b\,\mathrm{d}|\Xi|+a|\F|(\mathcal{C}).
		\qedhere
		\end{equation*}
	\end{proof}
	\begin{rem}[Monotonicity of Beckmann problem]\label{rem:monotonicityBeckmann}
	An immediate consequence is $B_{\tilde S,\tilde a,\tilde b}\geq B_{S,a,b}$ for any $\tilde S\subset S$, $\tilde a\geq a$, $\tilde b\geq b$.
	\end{rem}
	We can now prove strong duality for the Fenchel problems $W_d(\mu_+,\mu_-)$ (primal) and $B(\mu_+,\mu_-)$ (dual) using \cite[Thm.\,4.4.18, second constraint qualification]{BV} under \cref{assumpmetric}, from which we will afterwards deduce the result under \cref{afinite}. 
	\begin{prop}[{{Dual formulation of Beckmann problem $\widehat{=}$ \cref{BWfinal} under \cref{assumpmetric}}}]
		\label{Fenchel}
		Under \cref{assumpmetric} we have
		\begin{equation*}
		B(\mu_+,\mu_-)=\sup_{\varphi\in C_d^1}\int_{\mathcal{C} }\varphi\,\mathrm{d}(\mu_+-\mu_-)=W_d(\mu_+,\mu_-).
		\end{equation*}
	\end{prop}
	\begin{proof}
		We want to apply Fenchel's duality theorem \cite[Thm.\,4.4.18]{BV}.
		As usual we extend $b$ to $\mathcal C\setminus S$ by $a$.
		Consider the Banach spaces $X=C^1(\mathcal{C} )$ and $Y=C(\mathcal{C} ;\R^n)$ equipped with $|\varphi |_X=|\varphi|_{\infty ,\mathcal{C}}+|\nabla\varphi|_{\infty,\mathcal{C} }$ for $\varphi\in X$ and $|s|_Y=|s|_{\infty ,\mathcal{C}}$ for $s\in Y$. Define the mappings
		\begin{align*}
		f(\varphi)&=\langle\varphi,\mu_+-\mu_-\rangle&\text{ for }\varphi\in X,\\
		g(s)&=\iota_{\{|.|\leq b\text{ in }\mathcal{C}\}}(s)&\text{ for }s\in Y,\\
		A\varphi&=-\nabla\varphi&\text{ for }\varphi\in X.
		\end{align*}
		Clearly, $f$ and $g$ are convex and lower semi-continuous. Furthermore, $A$ is linear and bounded by $|A\varphi|_Y\leq|\varphi|_X$ for all $\varphi\in X$. Hence, by Fenchel's duality theorem \cite[Thm.\,4.4.18]{BV}
		\begin{equation*}
		\inf_{\varphi\in X}f(\varphi)+g(A\varphi)\geq\sup_{\F\in Y^*}-f^*(A^*\F)-g^*(-\F),
		\end{equation*}
		where $Y^*=\mathcal{M}^n(S)\times\mathcal{M}^n(\mathcal{C})$. In addition, we get
		\begin{equation*}
		f^*(\mu )=\sup_{\varphi\in X}\langle\varphi,\mu\rangle-f(\varphi)=\sup_{\varphi\in X}\langle\varphi,\mu-\mu_++\mu_-\rangle=\iota_{\{ \mu_+-\mu_-\}}(\mu)
		\end{equation*}
		for all $\mu\in X^*\supset\mathcal{M}^n(\mathcal{C})$. For $\varphi\in X$ and $\F\in Y^*$ we obtain
		\begin{equation*}
		\langle\varphi,A^*\F \rangle=\langle A\varphi,\F\rangle=-\langle\nabla\varphi,\F\rangle=\langle\varphi,\textup{div}\,\F\rangle.
		\end{equation*}
		We now calculate $g^*$.
		%For any $\F\in\mathcal{M}^n(\mathcal{C} )$ we observe
		%\begin{equation*}
		%k^*(\F)=\sup_{|t|_{\infty,\mathcal{C}}\leq a}\langle t,\F\rangle=\sup_{|t|_{\infty,\mathcal{C}}\leq 1}\langle at,\F\rangle=a|\F|(\mathcal{C} ).
		%\end{equation*}
		Let $\F\in\mathcal{M}^n(\mathcal C)$. Invoking \cite[p.\,130]{San} there is some Borel measurable function $\xi:\mathcal C\to\R^n$ with $\F=\xi |\F|$ and $|\xi|=1$ $|\F|$-almost everywhere on $\mathcal C$. Thus,
		\begin{equation*}
		g^*(\F )=\sup_{|s|\leq b}\langle s,\F\rangle\leq\sup_{|s|\leq b}\int_\mathcal{C}|s|\mathrm{d}|\F|\leq \int_\mathcal{C} b\mathrm{d}|\F|.
		\end{equation*}
		For the reverse inequality let $b_k:\mathcal C\to [0,\infty)$ be a sequence of Lipschitz functions with $b_k\nearrow b$ pointwise monotonically in $\mathcal C$ (such a sequence exists by \cite[Box 1.5]{San} due to the lower semi-continuity of $b$). Also note that by definition of the total variation there exists a sequence $(\tilde{s}_i)\subset C(\mathcal C;\R^n)$ with $|\tilde{s}_i|_{\infty,\mathcal C}\leq 1$ and $\langle\tilde{s}_i,\F\rangle\to|\F|(\mathcal C)$ for $i\to\infty$. For fixed $k$ we now define another sequence $(s_i)\subset C(\mathcal C;\R^n)$ by $s_i=b_k\tilde{s}_i$ and estimate
		\begin{align*}
		|\langle s_i,\F\rangle-\langle b_k,|\F|\rangle|\leq \int_\mathcal{C}|s_i\cdot\xi-b_k|\mathrm{d}|\F|=\int_\mathcal{C}|b_k\tilde{s}_i\cdot\xi-b_k|\mathrm{d}|\F|\leq|b_k|_{\infty,S}\int_\mathcal{C} |1-\tilde{s}_i\cdot\xi|\mathrm{d}|\F|=|b_k|_{\infty,S}(|\F|(S)-\langle\tilde{s}_i,\F\rangle)
		\end{align*}
		using that $|\F|$-almost everywhere $1-\tilde{s}_i\cdot\xi\in [0,2]$ and thus $|1-\tilde{s}_i\cdot\xi|=1-\tilde{s}_i\cdot\xi$. Thus we have $\langle s_i,\F\rangle\to\langle b_k,|\F|\rangle$ for $i\to\infty$. Furthermore, $\langle b_k,|\F|\rangle\nearrow\int_\mathcal{C} b\mathrm{d}|\F|$ by the monotone convergence theorem so that we end up with
		\begin{equation*}
		g^*(\F)=\int_\mathcal{C} b\mathrm{d}|\F|.
		\end{equation*}
		The function $g$ is continous in $0$ by the assumption that $b\geq \inf b>0$. Additionally, we have $0\in\text{dom}(f)$ and thus $0\in A\,\text{dom}(f)$. By \cite[Thm.\,4.4.18, second constraint qualification]{BV} strong duality holds, i.\,e.,
		\begin{multline*}
		\inf_{\F\in Y^*}\int_\mathcal{C} b\mathrm{d}|\Xi|+a|\F|(\mathcal{C} )+\iota_{\{ \mu_+-\mu_-\}}(\textup{div}\,\F)
		=-\sup_{\F\in Y^*}-f^*(A^*\F)-g^*(-\F)
		=-\inf_{\varphi\in X}f(\varphi )+g(A\varphi)\\
		=-\inf_{\varphi\in X}\langle\varphi,\mu_+-\mu_-\rangle+\iota_{\{|.|\leq b\text{ on }\mathcal{C}\}}(-\nabla\varphi)
		=\sup_{\varphi\in C_d^1}\langle\varphi,\mu_+-\mu_-\rangle.
		\qedhere
		\end{multline*}
	\end{proof}
	The subsequent statements will be used in the proof of \cref{BWfinal} under \cref{afinite} at the end of this section. The first lemma is standard and uses the lower semi-continuity of $d$ (which under \cref{assumpmetric} is \cref{thm:lscMetric} and under \cref{afinite} follows from the continuity of $d$ \cite[Prop.\,2.2.3]{LSW}).
	\begin{lem}[Existence of minimizer for $W_d(\mu_+,\mu_-)$ (e.\,g.\ \protect{\cite[Thm.\,1.5]{San}})]
		\label{optplan}
		Under \cref{assumpmetric} or \cref{afinite} there exists an optimal transport plan $\pi\in\Pi(\mu_+,\mu_-)$ such that
		\begin{equation*}
		W_d(\mu_+,\mu_-)=\int_{\mathcal{C}\times\mathcal{C}}d(x,y)\,\mathrm{d}\pi(x,y).
		\end{equation*}
	\end{lem}
	We show $W_d(\mu_+,\mu_-)=B(\mu_+,\mu_-)$ using standard approximation techniques to change \cref{assumpmetric} to \cref{afinite}. Henceforth, let $S^N$ be an approximating sequence for $S$ (see definition in \cref{notdefs}). We write $d^N=d_{S^N,a,b}$. Additionally, we use $b_\lambda=\max\{\lambda,b \}$ and $d_\lambda=d_{S,a,b_\lambda}$ for $\lambda\in (0,a)$. 
	\begin{lem}[Pointwise convergence of the $d^N$ and $d_\lambda$]
		\label{Siconv}
		We have $d_{\lambda}\searrow d$ pointwise as $\lambda\to0$.
		If \cref{afinite} is satisfied, then additionally $d^N\searrow d$ pointwise as $N\to\infty$.% Moreover, we have $d_{\lambda}\searrow d$ pointwise even without \cref{afinite}.
		%Under \cref{afinite} we have $W_{d^N}(\mu_+,\mu_-)\to \langle d,\pi\rangle$ for some $\pi\in\Pi(\mu_+,\mu_-)$.
	\end{lem}
	\begin{proof}
	Let $x,y\in\mathcal{C}$ with $d(x,y) <\infty$ (else there is nothing to show since $d_\lambda,d^N\geq d$). Let $\varepsilon>0$ and $\eta\in\Gamma^{xy}$ with $L(\eta)\leq d(x,y)+\varepsilon/2$. This yields
	\begin{equation*}
	|d_\lambda(x,y)-d(x,y)|\leq\inf_\gamma\int_{\gamma\cap S}b_\lambda\mathrm{d}\mathcal{H}^1+a\mathcal{H}^1(\gamma\backslash S)-\int_{\eta\cap S}b\mathrm{d}\mathcal{H}^1-a\mathcal{H}^1(\eta\backslash S)+\frac{\varepsilon}{2}\leq\int_{\eta\cap S\cap\{ \lambda>b \}}(\lambda-b)\mathrm{d}\mathcal{H}^1+\frac{\varepsilon}{2}<\varepsilon
	\end{equation*}
	for $\lambda$ small enough by choosing $\gamma=\eta$ in the second inequality. Note that all terms in the estimate are finite. Moreover, the last inequality followed from $\mathcal{H}^1(\eta\cap S\cap\{ \lambda>b \})\to 0$ for $\lambda\to 0$. This proves that $d_\lambda\searrow d$ pointwise. Assume now that \cref{afinite} is satisfied. Then we obtain
	\begin{equation*}
	|d^N(x,y)-d(x,y)|\leq \inf_\gamma\int_{\gamma\cap S^N}b\mathrm{d}\mathcal{H}^1+a\mathcal{H}^1(\gamma\backslash S^N)-\int_{\eta\cap S}b\mathrm{d}\mathcal{H}^1-a\mathcal{H}^1(\eta\backslash S)+\frac{\varepsilon}{2}\leq a\mathcal{H}^1(\eta\cap (S\backslash S^N))+\frac{\varepsilon}{2}<\varepsilon
	\end{equation*}
	for $N$ sufficiently large (again by choosing $\gamma=\eta$). The last inequality is true, because $\eta\cap S^N$ is an approximating sequence for $\eta\cap S$.
	\end{proof}
	\begin{examp}[$d^N\searrow d$ in general not true without \cref{afinite}]
	\label{pointwnottrue}
	Let $a=\infty,b\in[0,\infty)$, and $S=[0,1]$. Assume that $(I_j)$ is a sequence of non-empty pairwise disjoint intervals with $\bigcup_jI_j=S$. Set $S^N=\bigcup_{j=1}^NI_j$ (see \cref{notpointw}). Then we have $d^N(0,1)=\infty>b=d(x,y)$ for all $N$.
	\end{examp}
	\begin{figure}
		\centering
			\begin{tikzpicture}[scale=4]
%			\foreach \x in {0,1,...,100}
%				\node[circle,fill=black,inner sep=0.3pt,minimum size=0.01cm] at ({\x/100},0) {};
%			\foreach \y in {0,1,...,100}
%				\node[circle,fill=black,inner sep=0.3pt,minimum size=0.01cm] at ({\y/100},0.3) {};
%			\draw [color=black, thick, domain=-0.1:1] plot (\x,{1}) node[right] {$a$};
%			\draw[-stealth] (-0.1,0) -- (1.1,0);
%   			\draw[-stealth] (0,-0.1) -- (0,1.1) node[left] {$b$};
%   			\draw (1,-0.01) -- (1,0.01) node[below] {$1$};
			\draw (-0.1,0) -- (1.1,0);
			\draw (1,-0.01) -- (1,0.01) node[below] {$1$};
			\draw (0,-0.01) -- (0,0.01) node[below] {$0$};
			\node[label={[label distance=-0.1cm,text=gray]-90:  $S_7$}] at (0.5,0) {};
			\draw[line width=0.5mm,gray] (0,0) -- (0.1,0);
			\draw[line width=0.5mm,gray] (0.15,0) -- (0.21,0);
			\draw[line width=0.5mm,gray] (0.37,0) -- (0.39,0);
			\draw[line width=0.5mm,gray] (0.42,0) -- (0.48,0);
			\draw[line width=0.5mm,gray] (0.57,0) -- (0.66,0);
			\draw[line width=0.5mm,gray] (0.81,0) -- (0.84,0);
			\draw[line width=0.5mm,gray] (0.9,0) -- (0.98,0);
			\end{tikzpicture}
			\caption{Sketch for \cref{pointwnottrue}.}
			\label{notpointw}
		\end{figure}
%\notinclude
	\begin{proof}[Proof of \cref{BWfinal} under \cref{afinite}]
		\underline{$W_d(\mu_+,\mu_-)\leq B(\mu_+,\mu_-)$:} Initially, we show the inequality for the case $\inf b>0$. Let $\delta>0$ and fix $\xi$ and $\F^\perp$ as in the Beckmann problem. We define measures by $\G_N=\xi\Ha\mres (S\setminus S^N)+\F^\perp$. The function $|\xi|$ is integrable with respect to $\Ha\mres S$. Consequently, we have
		\begin{equation*}
		|\G_N-\F^\perp|(\mathcal{C} )=\int_{S\setminus S^N}|\xi|\,\mathrm{d}\Ha\to 0
		\end{equation*}
		for $N\to\infty$. Thus, we can choose $N$ sufficiently large such that $|\G_N-\F^\perp|(\mathcal{C})<\delta/a$.
		%By \cref{Siconv} we have $W_{d^N}(\mu_+,\mu_-)\to\langle d,\pi\rangle$ for some $\pi\in\Pi(\mu_+,\mu_-)$. We choose $N$ sufficiently large such that $|W_{d^N}(\mu_+,\mu_-)-\langle d,\pi\rangle|<\delta/2$.
		Using in this order $\textup{div}(\xi\Ha\mres  S^N+\G_N)=\textup{div}(\xi\Ha\mres S+\F^\perp)=\mu_+-\mu_-$, \cref{Fenchel}, \cref{KRformula}, and $d^N\geq d$, we can estimate
		\begin{multline*}
		\int_Sb|\xi|\,\mathrm{d}\Ha+a|\F^\perp|(\mathcal{C} )
		\geq \int_{S^N}b|\xi|\,\mathrm{d}\Ha+a|\G_N|(\mathcal{C} )-\delta
		%\geq \inf_{\F\in\mathcal{M}^n(\mathcal{C})} \int_{S^N}b\,\mathrm{d}|\F\mres S^N|+a|\F|(\mathcal{C}\setminus S^N )+\iota_{\{ \mu_+-\mu_- \}}(\textup{div}\,\F)-\delta\\
		\geq B_{S^N,a,b|_{S^N}}(\mu_+,\mu_-)-\delta\\
		=\sup_{\varphi\in C_{d^N}^1}\int_{\mathcal{C} }\varphi\,\mathrm{d}(\mu_+-\mu_-)-\delta
		=W_{d^N}(\mu_+,\mu_-)-\delta
		%=\langle d^N,\pi_N\rangle-\delta/2
		%\geq \langle d,\pi\rangle-\delta
		\geq W_d(\mu_+,\mu_-)-\delta,
		\end{multline*}
		where $C_{d^N}^1=\{ \varphi\in C^1(\mathcal{C})\,|\,|\nabla\varphi|\leq b\text{ on }S^N,|\nabla\varphi|\leq a\text{ in }\mathcal{C} \}$. Letting $\delta\to 0$ yields the desired inequality. Now we concentrate on the case $\inf b=0$. Using the functions $b_\lambda$ we obtain
		\begin{equation*}
		\int_Sb_\lambda|\xi|\,\mathrm{d}\Ha+a|\F^\perp|(\mathcal{C} )\geq W_d(\mu_+,\mu_-).
		\end{equation*} 
		Moreover, we have 
		\begin{equation*}
		\int_Sb_\lambda|\xi|\,\mathrm{d}\Ha\leq a\int_S|\xi|\,\mathrm{d}\Ha<\infty
		\end{equation*}
		and therefore (monotone convergence)
		\begin{equation*}
		\int_Sb|\xi|\,\mathrm{d}\Ha+a|\F^\perp|(\mathcal{C} )\geq W_d(\mu_+,\mu_-)
		\end{equation*}
		by letting $\lambda\to 0$.
		\\\underline{$W_d(\mu_+,\mu_-)\geq B(\mu_+,\mu_-)$:} First assume that $\inf b>0$. By \cref{optplan} there exists an optimal transport plan $\pi\in\Pi(\mu_+,\mu_-)$ such that $W_d(\mu_+,\mu_-)=\langle d,\pi\rangle$. Let $\delta>0$ be arbitrary. Using $\langle d^N,\pi\rangle<\infty$ and $d^N\searrow d$ pointwise by \cref{Siconv} the monotone convergence theorem implies the existence of some $N=N(\delta)$ such that $|\langle d^N-d,\pi\rangle|\leq\delta$. Application of \cref{Fenchel} and \cref{rem:monotonicityBeckmann} yields
		\begin{equation*}
		W_d(\mu_+,\mu_-)=\langle d,\pi \rangle
		\geq \langle d^N,\pi\rangle-\delta
		\geq W_{d^N}(\mu_+,\mu_-)-\delta
		=B_{S^N,a,b|_{S^N}}(\mu_+,\mu_-)-\delta
		%=\inf_{\xi,\F^\perp}\int_{S^N}b|\xi|\,\mathrm{d}\Ha+a|\F^\perp|(\mathcal{C} )+\iota_{\{ \mu_+-\mu_-\}} (\xi\Ha\mres S^N+\F)-\delta\\
		%&\geq\inf_{\hat{\xi},\hat{\F}^\perp }\int_{S}b|\hat{\xi}|\,\mathrm{d}\Ha+a|\hat{\F}^\perp|(\mathcal{C} )+\iota_{\{ \mu_+-\mu_-\}} (\hat{\xi}\Ha\mres S+\hat{\F}^\perp )-\delta,
		\geq B_{S,a,b}(\mu_+,\mu_-)-\delta.
		\end{equation*}
		%where the infima are taken over functions $\xi,\F^\perp,\hat{\xi},\hat{\F}^\perp$ as in \cref{subS} with $T=S^N$.
		Letting $\delta\to 0$ yields the desired inequality. Now assume that $\inf b=0$.
		Once more let $\delta>0$. Again, by $\langle d_{\lambda},\pi\rangle<\infty$  and $d_{\lambda}\searrow d$ pointwise by \cref{Siconv} we can apply the monotone convergence theorem and choose $\lambda$ sufficiently small such that
		\begin{equation*}
		W_d(\mu_+,\mu_-)=\langle d,\pi\rangle
		\geq \langle d_{\lambda},\pi\rangle-\delta
		\geq W_{d_{\lambda}}(\mu_+,\mu_-)-\delta
		\geq B_{S,a,b_\lambda}(\mu_+,\mu_-)-\delta
		\geq B_{S,a,b}(\mu_+,\mu_-)-\delta.
		\end{equation*}
		Letting $\delta\to 0$ we obtain the desired inequality.
	\end{proof}
	\section{Urban planning formulation of branched transport using duality}
	\label{uppeqbtp}
	In this section we prove \cref{finalthm} under \cref{growthcond}. We will use the following formula for the branched transport problem \cite[Lemma\,3.1.8]{LSW}, which can be derived from \cite[Lem.\,5.15]{BCM}, \cite[Prop.\,2.32]{BW}, and \cite[Thm.\,3.1]{Sil}.
	\begin{lem}[Version of the branched transport problem]
		\label{version}
		The branched transport problem can be written as
		\begin{equation*}
		\inf_{\F\in\mathcal{DM}^n(\mathcal{C}) }\mathcal{J}^{\tau,\mu_+,\mu_-}[\F]=\inf_{S,\xi,\F^\perp}\int_S\tau (|\xi|)\,\mathrm{d}\Ha+\tau'(0)|\F^\perp|(\mathcal{C})+\iota_{\{\mu_+-\mu_-\}}(\textup{div}(\xi\Ha\mres S+\F^\perp))
		\end{equation*}
		with $S\subset \mathcal{C}$ countably $1$-rectifiable and Borel measurable, $\xi\in L^1(\Ha\mres S;\R^n)$ and $\F^\perp\in\mathcal{M}^n(\mathcal{C})$ with $\F^\perp\mres S=0$.
	\end{lem}
	In \cref{subs32} we will prove a dual formula for the first network cost term, which will be used to introduce the friction coefficients $b$. Based on this formula and \cref{BWfinal}, we will provide a short proof of \cref{finalthm} under \cref{growthcond}.
	\subsection{Duality for the total network transportation cost}
	\label{subs32}
	We apply a generalization of Rockafellar's duality theorem between a pair of decomposable spaces of vector-valued functions \cite[Thm.\,VII-7]{CV}. Let $S\subset\mathcal{C}$ be countably $k$-rectifiable and $\mathcal{H}^k$-measurable ($k=1$ and $S$ even Borel in our setting). We denote the $\sigma$-algebra of $\mathcal{H}^k$-measurable subsets of $S$ (or more precisely subsets that are Carath\'{e}odory-measurable relative to $\mathcal{H}^k$) by $\mathscr{H}^k(S)$. Using that the restriction of an outer measure like $\mathcal{H}^k$ on $\R^n$ to the Carath\'{e}odory-measurable sets yields a complete measure space, this property also holds for $(S,\mathscr{H}^k(S),\mathcal{H}^k\mres S)$. Furthermore, the measure space $(S,\mathscr{H}^k(S),\mathcal{H}^k\mres S)$ is $\sigma$-finite (every member of an approximating sequence for $S$ has finite measure). 
	\begin{conv}
		In this section ``measurability'' and ``integrability'' will refer to the measure space $(S,\mathscr{H}^k(S),\mathcal{H}^k\mres S)$.
	\end{conv}
	We remind the reader that a subset $A\subset S$ is in $\mathscr{H}^k(S)$ if and only if $\mathcal{H}^k(B)=\mathcal{H}^k(B\cap A)+\mathcal{H}^k(B\backslash A)$ for all $B\subset S$. By Carath\'{e}odory's criterion \cite[Thm.\,1.13]{LS} we have $\mathcal{B}(S)\subset\mathscr{H}^k(S)$.
	In our case, the ``vector-valued functions'' from \cite[Chap.\,VII-7]{CV} are given through integrable and real-valued functions on $S$. We consider the following vector spaces:
	\begin{align*}
	\mathscr{L}&=\{ f:S\to\R\text{ measurable and integrable} \} ,\\
	\mathscr{L}_{\text{b}}&=\{ f\in\mathscr{L}\, |\, f\text{ bounded}\},\\
	\mathscr{M}_{\text{b}}&=\{ f:S\to\R\text{ measurable and bounded}\}.
	\end{align*}
	Note that the elements of $\mathscr{L},\mathscr{L}_{\text{b}}$, and $\mathscr{M}_{\text{b}}$ are not equivalence classes. We follow here the setting in \cite[Def.\,VII-3]{CV}.
	\begin{defin}[{{Decomposable \cite[Def.\,VII-3]{CV}}}]
		\label{defdec}
		Let $V\subset\mathscr{L}$ be a subspace. The space $V$ is said to be decomposable if $1_Af+1_{S\setminus A}g\in V$ for all $A\in\mathscr{H}^k(S)$ with $\mathcal{H}^k(A)<\infty$, $f\in\mathscr{M}_{\text{b}}$ and $g\in V$.
	\end{defin}
	The vector spaces $\mathscr{L}$ and $\mathscr{L}_{\text{b}}$ are decomposable: Let either $V=\mathscr{L}$ or $V=\mathscr{L}_{\text{b}}$ and $A,f,g$ as in \cref{defdec}. Then we have
	\begin{equation*}
	\int_S|1_Af+1_{S\setminus A}g|\,\mathrm{d}\mathcal{H}^k\leq\mathcal{H}^k(A)|f|_{\infty,A}+\int_{S\setminus A}|g|\,\mathrm{d}\mathcal{H}^k<\infty
	\end{equation*}
	and thus $1_Af+1_{S\setminus A}g\in\mathscr{L}$. Further, if $V=\mathscr{L}_{\text{b}}$, then we obtain $|1_Af+1_{S\setminus A}g|\leq |f|_{\infty,A}+|g|_{\infty,A}<\infty$. Moreover, we get $x\mapsto f(x)g(x)\in\mathscr{L}$ for all $f\in\mathscr{L},g\in \mathscr{L}_{\text{b}}$ (cf.\ \cite[Def.\,VII-3]{CV}), which needs to be satisfied to apply \cite[Thm.\,VII-7]{CV}. 
	%\begin{rem}
	%	\label{cvrem}
	%	The setting in \cite[Chapter VII-3]{CV}) is much more general. One may replace $(S,\mathscr{H}^k(S),\mathcal{H}^k)$ by any complete and $\sigma$-finite measure space. Moreover, $\R$ could be replaced by any locally convex space $E$ such that $E$ and its topological dual $E'$ are Suslin locally convex spaces which are compatible with duality (cf.\ \cite[p.\,195]{CV}). One then considers scalarly integrable functions with values in $E$ (respectively $E'$) and the above spaces become a lot more general \cite[Def.\,VII-3]{CV}. We will also consider lower semi-continuous (and thus Borel measurable) functions $f:\R\to (-\infty ,\infty]$ which also may be replaced by so called ``normal integrands'' (cf.\ \cite[Def.\,VII-1]{CV}).
	%\end{rem}
	For $g\in\mathscr{L}$ and lower semi-continuous functions $f:\R\to \R\cup\{ \infty \}$ we write
	\begin{equation*}
	I_f(g)=\begin{cases}
	\int_Sf(g)\,\mathrm{d}\mathcal{H}^k& \text{if }\int_Sf(g)^+\,\mathrm{d}\mathcal{H}^k<\infty,\\
	\infty&\text{else},
	\end{cases}
	\end{equation*} 
	where $f(g)^+$ denotes the positive part of $f(g):S\to  \R\cup\{ \infty \}$. By our construction we can apply the following statement to generate the friction coefficients $b:S\to[0,\infty)$ of the urban planning problem.
	Essentially it characterizes the convex conjugate of integral functionals using $L^1$-$L^\infty$-type duality pairings: One may pull the convex conjugation inside the integral.
	\begin{prop}[{{First part of \cite[Thm.\,VII-7]{CV}}}]
		\label{cv}
		Let $f:\R\to \R\cup\{ \infty \}$ be lower semi-continuous. Assume that $I_f(g_0)$ is finite for at least one $g_0\in\mathscr{L}_{\text{b}}$. Then we obtain
		\begin{equation*}
		I_{f^*}(h)=\sup_{g\in\mathscr{L}_{\text{b}}}\int_Sgh\,\mathrm{d}\mathcal{H}^k-I_f(g)
		\end{equation*}
		for all $h\in\mathscr{L}$.
	\end{prop}
	%\begin{rem}
	%	Recall that $f^*$ is convex and lower semi-continuous \cite[Corollary VII-2]{CV}. 
	%\end{rem}
	Fix a maintenance cost $\varepsilon$ induced by a transportation cost $\tau$, i.\,e., $\varepsilon(b)=(-\tau)^*(-b)$, where $\tau(m)=-\infty$ for $m<0$. Then we get the following formula for the network cost term from \cref{version}.
	\begin{corr}[Substitution of maintenance cost]
		\label{subst3}
		Assume that $\mathcal{H}^k(S)<\infty$. Furthermore, let $\xi :S\to\R^n$ be integrable. Let the transportation cost $\tau$ be right-continuous in $0$. Then we have
		\begin{equation*}
		\int_S\tau (|\xi |)\,\mathrm{d}\mathcal{H}^k=\inf_b\int_Sb|\xi |\,\mathrm{d}\mathcal{H}^k+\int_S\varepsilon (b)\,\mathrm{d}\mathcal{H}^k,
		\end{equation*}
		where the infimum is taken over $b\in\mathscr{L}_{\text{b}}$.% with $c=\inf\text{dom}(\varepsilon)\leq b\leq\tau'(0)$.
	\end{corr}
	\begin{proof}
		The biconjugate equals the convex and lower semi-continuous envelope \cite[Prop.\,2.28]{R}. The function $-\tau$ is already convex and lower semi-continuous by assumption. Thus, we conclude
		\begin{equation*}
		\tau (m)=-(-\tau )^{**}(m)=-\left(\sup_{v\in\R}vm-(-\tau )^*(v)\right)=-\left(\sup_{v\in\R}-vm-\varepsilon (v)\right)=-\varepsilon^*(-m).
		\end{equation*}
		Furthermore, we have $I_{\varepsilon}(v)<\infty$ for all $v\in\R$ with $\varepsilon (v)<\infty$ using $\mathcal{H}^k(S)<\infty$. We presumed $|\xi|\in\mathscr{L}$ and can therefore apply \cref{cv}:
		\begin{multline*}
		\int_S\tau (|\xi|)\,\mathrm{d}\mathcal{H}^k=-\int_S\varepsilon^*(-|\xi|)\,\mathrm{d}\mathcal{H}^k=-I_{\varepsilon^*}(-|\xi|)=-\sup_{b\in\mathscr{L}_{\text{b}}}-\int_Sb|\xi|\,\mathrm{d}\mathcal{H}^k-I_\varepsilon (b)=\inf_{b\in\mathscr{L}_{\text{b}}}\int_Sb|\xi|\,\mathrm{d}\mathcal{H}^k+\int_S\varepsilon (b)\,\mathrm{d}\mathcal{H}^k.\qedhere
		\end{multline*}
		%		We can require $b\geq c$, because $\varepsilon (v)=\infty$ for $v<c$. Moreover, if $b\in\mathscr{L}_{\text{b}}$ with $b\geq c$ such that the expression on the right-hand side is finite, then $\tilde{b}=\min\{ b,\tau'(0) \}$ satisfies $\tilde{b}\leq\tau'(0)$ and
		%		\begin{equation*}
		%		\int_S\tilde{b}|\xi|\,\mathrm{d}\mathcal{H}^k+\int_S\varepsilon (\tilde{b})\,\mathrm{d}\mathcal{H}^k\leq \int_Sb|\xi|\,\mathrm{d}\mathcal{H}^k+\int_S\varepsilon (\tilde{b})\,\mathrm{d}\mathcal{H}^k=\int_Sb|\xi|\,\mathrm{d}\mathcal{H}^k+\int_S\varepsilon (b)\,\mathrm{d}\mathcal{H}^k
		%		\end{equation*}
		%		by using $\varepsilon(v)=(-\tau )^*(-v)=0$ for $v\geq \tau'(0)$. 
	\end{proof}
	We now prove \cref{subst} in a slightly generalized setting (for arbitrary $k$ and $S$ not necessarily Borel measurable).
	\begin{thm}[Generalization of \cref{subst}]
		\label{subst2}
		Let $\xi :S\to\R^n$ be integrable. Assume that $\tau $ satisfies \cref{growthcond}. Then we have
		\begin{equation*}
		\int_S\tau (|\xi |)\,\mathrm{d}\mathcal{H}^k=\inf_b\int_Sb|\xi |\,\mathrm{d}\mathcal{H}^k+\int_S\varepsilon (b)\,\mathrm{d}\mathcal{H}^k,
		\end{equation*}
		where the infimum is taken over lower semi-continuous functions $b: S\to [0,\tau'(0)]$.
	\end{thm}
	\begin{proof}
		\underline{Formula true for functions $b\in\mathscr{M}_b$:}
		Using the procedure in \cref{notdefs} (compare with the sets $T_i$), we can assume that $S$ is a disjoint countable union of finitely $k$-rectifiable sets $S_i$, i.\,e., the $S_i\subset S$ are measurable with $\mathcal{H}^k(S_i)<\infty$ (we neglect nullsets due to the integration). Let $\delta>0$ and $\delta_i>0$ such that $\sum_i\delta_i=\delta$. By \cref{subst3} we can find functions $b_i\in\mathscr{L}_b$ with
		\begin{equation*}
		\int_{S_i}\tau (|\xi |)\,\mathrm{d}\mathcal{H}^k\geq\int_{S_i}|\xi|b_i+\varepsilon (b_i)\,\mathrm{d}\mathcal{H}^k-\delta_i.
		\end{equation*}
		Define $\tilde{b}\in\mathscr{M}_\text{b}$ by $\tilde{b}=b_i$ on $S_i$. Note that we can assume that each $b_i$ is bounded by $\tau'(0)$ because of $\varepsilon([\tau'(0),\infty))=\{ 0\}$. Application of the monotone convergence theorem yields
		\begin{multline*}
		\int_S\tau (|\xi|)\,\mathrm{d}\mathcal{H}^k
		= \sum_i\int_{S_i}\tau (|\xi|)\,\mathrm{d}\mathcal{H}^k
		\geq \sum_i\int_{S_i}|\xi|b_i+\varepsilon (b_i)\,\mathrm{d}\mathcal{H}^k-\delta
		=\int_{S}|\xi|\tilde{b}+\varepsilon (\tilde{b})\,\mathrm{d}\mathcal{H}^k-\delta\\
		\geq\inf_b\int_{S}|\xi|b+\varepsilon (b)\,\mathrm{d}\mathcal{H}^k-\delta
		=\inf_b\int_{S}|\xi|b\,\mathrm{d}\mathcal{H}^k+\int_S\varepsilon (b)\,\mathrm{d}\mathcal{H}^k-\delta,
		\end{multline*}
		where the infima are taken over functions $b\in\mathscr{M}_\text{b}$. Notice that $\varepsilon,b\geq 0$ was used in the last equation. The reverse inequality is a direct consequence of the definition of the convex conjugate. More precisely, let $b:S\to [0,\infty)$ be any measurable and bounded function. Then we obtain
		\begin{equation*}
		\int_S\tau (|\xi|)\,\mathrm{d}\mathcal{H}^k=-\int_S\varepsilon^*(-|\xi|)\,\mathrm{d}\mathcal{H}^k=-\int_S\sup_{v\in\R}-v|\xi (x)|-\varepsilon (v)\,\mathrm{d}\mathcal{H}^k(x)\leq\int_Sb|\xi|\,\mathrm{d}\mathcal{H}^k+\int_S\varepsilon(b)\,\mathrm{d}\mathcal{H}^k
		\end{equation*}
		by choosing $v=v(x)=b(x)$ in the inequality. Taking the infimum over $b$ shows the stated formula for $b\in\mathscr{M}_\text{b}$.
		\mbox{}
		\\\underline{The functions $b$ can be assumed to be lower semi-continuous:} As a first step, we simplify the statement to be shown. Write $c=\inf\text{dom}(\varepsilon)$. Note that every function in the infimum of the stated formula of the theorem can be assumed to have values in $[c,\tau'(0)]$ by the properties of $\varepsilon$. For $f\in\mathscr{M}_{\text{b}}$ with $f\geq 0$ define
		\begin{equation*}
		F(f)=\int_Sf|\xi|+\varepsilon (f)\,\mathrm{d}\mathcal{H}^k\in [0,\infty ].
		\end{equation*}
		Let $b\in\mathscr{M}_{\text{b}}$ be arbitrary with $F(b)<\infty$ (note that $F(\tau'(0))<\infty$ by the integrability of $|\xi|$ and $\varepsilon(\tau'(0))=0$). Let $S^N$ be an approximating sequence for $S$. Define functions by
		\begin{equation*}
		b_N=\begin{cases}
		b&\text{on }S^N,\\
		\tau'(0)&\text{on }S\setminus S^N.
		\end{cases}
		\end{equation*}
		We have $\varepsilon (b_N)\nearrow\varepsilon (b)$ and $b_N|\xi|\searrow b|\xi|$ for $N\to\infty$ pointwise $\mathcal{H}^k\mres S$-almost everywhere. Furthermore, the $b_N|\xi|$ satisfy
		\begin{equation*}
		\int_Sb_N|\xi|\,\mathrm{d}\mathcal{H}^k\leq\tau'(0)\int_S|\xi|\,\mathrm{d}\mathcal{H}^k<\infty. 
		\end{equation*}
		Hence, we get $F(b_N)\to F(b)$ for $N\to\infty$ by the monotone convergence theorem. Thus, for fixed $N$ it is enough to find a lower semi-continuous funcion $\ell_N:S\to [c,\tau'(0)]$ such that $F(\ell_N)$ is arbitrarily close to $F(b_N)$. This reduces to the problem of finding a lower semi-continuous function $\ell_N:S^N\to [c,\tau'(0)]$ with $F_N(\ell_N)$ being arbitrarily close to $F_N(b_N)$, where
		\begin{equation*}
		F_N(f)=\int_{S^N}f|\xi|+\varepsilon (f)\,\mathrm{d}\mathcal{H}^k
		\end{equation*}
		for $f\in\mathscr{M}_{\text{b}}$ with $f\geq 0$, because we can then replace $\ell_N$ by the lower semi-continuous function ($S^N$ is closed)
		\begin{equation*}
		\tilde{\ell}_N=\begin{cases}
		\ell_N&\text{on }S^N,\\
		\tau'(0)&\text{on }S\setminus S^N.
		\end{cases}
		\end{equation*}
		Note that $F_N(b_N)$ is finite by $F(b)<\infty$. For simplicity, we write $S,F,b$ instead of $S^N,F_N,b_N$ and assume that $S$ is closed with $\mathcal{H}^k(S)<\infty$ (like the $S^N$) and $F(b)<\infty$ (which implies $\varepsilon(b)\in\mathscr{L}$) for the rest of the proof. By \cite[Rem.\,2.50]{AmFuPa00} (which refers to \cite[Ex.\,1.3]{AmFuPa00}) the function $b\in\mathscr{M}_\text{b}$ can be assumed to be Borel measurable. We will below use the functions $b_\delta=b+\delta$ for $\delta>0$. Without loss of generality the $b_\delta$ have values in $[c+\delta,\tau'(0)+\delta]$. Since the restriction of $\mathcal{H}^k\mres S$ to $\mathcal{B}(S)$ now is a finite Radon measure on a compact set, we can apply the Vitali\textendash Carath\'{e}odory theorem \cite[\nopp 2.25]{Rud} to the functions $b_\delta\in\mathscr{L}$: for all $\delta>0$ there exists a sequence of lower semi-continuous functions $\ell_\delta^i$ with $b_\delta \leq \ell_\delta^i\leq\tau'(0)+\delta$ on $S$ and
		\begin{equation*}
		\int_S(\ell_\delta^i-b_\delta)\,\mathrm{d}\mathcal{H}^k\to 0\text{ for }i\to\infty.
		\end{equation*}
		This yields $\ell_\delta^i\to b_\delta$ pointwise $(\mathcal{H}^k\mres S)$-almost everywhere for $i\to\infty$ up to a subsequence. Now let $f_\delta^i=\ell_\delta^{i}|\xi|+\varepsilon (\ell_\delta^{i})$. We obtain $|f_\delta^{i}|\leq (\tau'(0)+\delta)|\xi|+\varepsilon (c+\delta)\in\mathscr{L}$ and $f_\delta^{i}\to b_\delta|\xi|+\varepsilon (b_\delta )$ pointwise $\mathcal{H}^k\mres S$-almost everywhere for $i\to\infty$ ($\varepsilon$ is continuous on $(c,\infty )$). Thus, we have $F(\ell_\delta^{i})\to F(b_\delta)$ for $i\to\infty$ by Lebesgue's dominated convergence theorem. In addition, we have $h_\delta=b_\delta|\xi|+\varepsilon (b_\delta)\to b|\xi|+\varepsilon (b)$ pointwise for $\delta\to 0$ ($\varepsilon$ is right-continuous in $c$) and $|h_\delta|\leq (\tau'(0)+\delta)|\xi|+\varepsilon (b)\in\mathscr{L}$, which yields $F(b_\delta )\to F(b)$ for $\delta\to 0$. This shows that $b$ can be assumed to be lower semi-continuous (choose $\delta$ sufficiently small and then $i$ sufficiently large for the corresponding subsequence of the $\ell_\delta^{i}$).
	\end{proof}
	\begin{rem}[Construction of minimizer]
		\label{constrmin}
		If $|\xi|$ is upper semi-continuous, then a minimizer $b$ is given by (cf.\ proof of \cite[Theorem 1.3.4]{LSW})
		\begin{equation*}
		b=-\max\partial(-\tau)(|\xi|).
		\end{equation*}
	\end{rem}
	\subsection{Branched transport as street network optimization}
	\begin{proof}[Proof of \cref{finalthm} under \cref{growthcond}]
		We apply \cref{version}, \cref{subst}, and \cref{BWfinal} (the variables are as in the problems and statements):
		\begin{align*}
		\inf_\F\mathcal{J}^{\tau,\mu_+,\mu_-}[\F]&=\inf_{\xi,S,\F^\perp}\int_S\tau(|\xi|)\mathrm{d}\mathcal{H}^1+\tau'(0)|\F^\perp|(\mathcal{C})+\iota_{\{ \mu_+-\mu_-\}}(\textup{div}(\xi\Ha\mres S+\F^\perp))\\
		&=\inf_{\xi,S,\F^\perp}\inf_b\int_Sb|\xi|\mathrm{d}\mathcal{H}^1+\int_S\varepsilon(b)\mathrm{d}\mathcal{H}^1+\tau'(0)|\F^\perp|(\mathcal{C})+\iota_{\{ \mu_+-\mu_-\}}(\textup{div}(\xi\Ha\mres S+\F^\perp))\\
		&=\inf_{S,b}\inf_{\xi,\F^\perp}\int_Sb|\xi|\mathrm{d}\mathcal{H}^1+\tau'(0)|\F^\perp|(\mathcal{C})+\iota_{\{ \mu_+-\mu_-\}}(\textup{div}(\xi\Ha\mres S+\F^\perp))+\int_S\varepsilon(b)\mathrm{d}\mathcal{H}^1\\
		&=\inf_{S,b}W_{d_{S,\tau'(0),b}}(\mu_+,\mu_-)+\int_S\varepsilon(b)\mathrm{d}\mathcal{H}^1\\
		&=\inf_{S,b}\mathcal{U}^{\varepsilon,\mu_+,\mu_-}[S,b].\qedhere
		\end{align*}
	\end{proof}
	\begin{rem}[Branched transport problem has minimizer under \cref{growthcond}]
		The growth condition $\tau'(0)<\infty$ implies
		\begin{equation*}
		\int_{0}^{1}\frac{\tau(m)}{m^{2-1/n}}\mathrm{d}m\leq\tau'(0)\int_{0}^{1}\frac{1}{m^{1-1/n}}\mathrm{d}m=\tau'(0)n<\infty,
		\end{equation*}
		which by \cite[Corollaries 2.19 \& 2.20]{BW} guarantees that the branched transport problem is finite and has a solution $\F$.
	\end{rem}
	\section{Acknowledgements}
	This work was supported by the Deutsche Forschungsgemeinschaft (DFG, German Research Foundation) under the priority program SPP 1962, grant WI 4654/1-1, and under Germany's Excellence Strategy EXC 2044 -- 390685587, Mathematics M\"unster: Dynamics\textendash Geometry\textendash Structure. B.W.’s and J.L.'s research was supported by the Alfried Krupp Prize for Young University Teachers awarded by the Alfried Krupp von Bohlen und Halbach-Stiftung. B.S.'s research was supported by the Emmy Noether Programme of the DFG, grant SCHM 3462/1-1.
	\printbibliography
\end{document}